\documentclass[12pt,a4paper]{amsart}
\usepackage[utf8]{inputenc}
\usepackage{amsfonts}
\usepackage{amsmath, amssymb}
\usepackage{amsthm}
\usepackage{cases}
\usepackage[margin=.8in]{geometry}
\usepackage{tkz-euclide}
\usepackage{comment}
\usepackage{lineno}
\usepackage{enumitem}

\newtheorem{theorem}{Theorem}[section]

\newtheorem{defin}[theorem]{Definition}
\newtheorem{lemma}[theorem]{Lemma}
\newtheorem{prop}[theorem]{Proposition}
\newtheorem{rem}[theorem]{Remark}

 \usepackage{hyperref}
 \hypersetup{
 	colorlinks = true, 
 	urlcolor = blue, 
 	linkcolor = blue, 
 	citecolor = blue 
 }

\usepackage[notcite,notref,color]{showkeys} 
\definecolor{labelkey}{gray}{0.8} \definecolor{refkey}{gray}{0.8}

\let \div \relax

\DeclareMathOperator{\div}{div}

\DeclareMathOperator{\dist}{dist}

\DeclareMathOperator{\N}{\mathbb{N}}

\DeclareMathOperator{\R}{\mathbb{R}}

\providecommand{\Test}{\mathcal{D}}

\newcommand{\dd}{\, \mathrm{d}}
\newcommand{\del}{\partial}
\newcommand{\eps}{\varepsilon}
\newcommand{\weak}{\rightharpoonup}
\newcommand{\B}{\mathcal{B}}

\newcommand{\ms}{\mathbb{S}}
\newcommand{\vc}{\mathbf}
\newcommand{\vu}{\vc u}
\newcommand{\tvu}{\tilde{\vu}}
\newcommand{\tvr}{\tilde{\rho}}
\newcommand{\ovr}{{\overline{\rho}}}
\newcommand{\BMO}{{\rm BMO}}

\renewcommand{\rho}{\varrho}
\newcommand{\lr}[1]{\left( #1 \right)}
\let\temp\phi
\let\phi\varphi
\let\varphi\temp

\newcommand{\rhoe}{\varrho_{\eps}}

\allowdisplaybreaks

\title[Hom. NSE hard sphere pressure]{Homogenization of compressible Navier-Stokes equations under a hard sphere pressure law}
\date{\today}

\author{Nilasis Chaudhuri}
\address{Institute of Applied Mathematics and Mechanics, University of Warsaw, ul. Banacha 2, 02-097 Warszawa, Poland}
\email{nchaudhuri@mimuw.edu.pl}

\author{Florian Oschmann}
\address{Faculty of Mathematics and Physics of the Charles University, Sokolovsk\'a 49/83, 186 00 Praha 8, Czech Republic.}
\email{florian.oschmann@matfyz.cuni.cz}

\begin{document}
	
	\begin{abstract}
		We consider the compressible time-dependent Navier-Stokes equations in a bounded perforated domain in dimensions two and three. Provided the perforations are small enough, we show that the limiting equations do not change their form when the perforation size goes to zero while their number increases to infinity. The novelty of this result is the form of the pressure: we consider a hard-sphere pressure law, giving an \emph{a priori} bound for the density while, compared to the barotropic case, having worse regularity for the pressure, therefore causing significant problems in the homogenization procedure. To the best of our knowledge, the homogenization for this kind of pressures has not been addressed in the literature yet.
	\end{abstract}
	
	\maketitle
	\tableofcontents

	\section{Introduction}
	The rigorous derivation of effective equations from small scale equations resp. equations in perforated domains has gained a lot of interest during the last decades. This so-called homogenization procedures are ways how to derive the famous Darcy and Brinkman laws, or to recover the original equations one started with, depending on the pore size of the porous medium under consideration. A full review of the available literature is out of scope, so we will just cite a few and refer the interested reader to the references therein. Regarding incompressible fluids, Allaire gave in \cite{Allaire1990a, Allaire1990b} a comprehensive study, which was followed up also for different types of fluids and perforations by \cite{GiuntiHoefer2019, Giunti2021, HoeferHuebner2026, HoeferLuOschmann2026, Lu2020, LuQian2024, LuYang2023, WiedemannPeter2024, WiedemannPeter2025}. Incompressible inhomogeneous fluids with non-constant density have been recently considered in \cite{BasaricOschmannPan2026, LuPanYang2025, Pan2025}. Fully compressible fluids were treated in \cite{DieningFeireislLu2017, FeireislLu2015, NecasovaPan2022, OschmannPokorny2023}, among others. For the transition regime from compressible to incompressible fluids via a singular limit involving the low Mach number, we refer to \cite{BasaricChaudhuri2024, BellaOschmann2022, BellaOschmann2023, HoeferKowalczykSchwarzacher2021, HoeferNecasovaOschmann2026} and the references therein.\\

    All of the mentioned works for compressible fluids focus on the \emph{barotropic pressure law}, that is, the pressure is of the form $p(\rho)=\rho^\gamma$ for some $\gamma \geq 1$, where $\rho$ is the fluid's density. This pressure law obviously allows for arbitrarily high densities. Physically, however, fluids consist out of molecules of a specific size that can't be compressed forever; in other words, the fluid's density is bounded from above, and the pressure needs to blow up when the density reaches this upper limit to act against further compression, an idea that was established already by van der Waals in \cite[p. 56]{vanderWaals1873} and almost one hundred years later refined by Soave in \cite{Soave1972}. Such behaviour can be modeled by so-called \emph{hard-sphere pressure laws}. Our goal in this work is therefore to consider homogenization of a certain compressible fluid in the regime of such singular pressures, which, according to the best of the authors' knowledge, has not been investigated yet in the literature.\\
    
\paragraph{\bf Challenges for hard-sphere pressures} The major difference to the barotropic case is that the pressure potential, which pops up in the energy inequality and is defined by
\begin{align*}
    P(\rho) = \rho \int_1^\rho \frac{p(z)}{z^2} \dd z ,
\end{align*}
has another behaviour than $p$ alone. Indeed, if $p(\rho) \approx \rho^\gamma$, then also $P(\rho) \approx \rho^\gamma$ and uniform bounds on $p$ follow from those obtained for $P$. However, if $p(\rho) \approx (1- \rho)^{-\beta}$ for some $\beta >1$, then $P(\rho) \approx (1-\rho)^{-(\beta-1)}$, which is a weaker singularity and therefore no information on the bounds for $p$ can be extracted from this. Furthermore, to pass to the (weak) limit in the nonlinear function $p(\rho)$ is usually established via some higher integrability $p(\rho) \in L^q(D)$ for some $q>1$, whereas in the hard-sphere pressure case, the best one can get is $p(\rho) \in L^1(D)$ such that one needs additional integrability properties independent of the size of the perforations. We aim to find such uniform bounds and integrability properties to still be able to pass to the limit in the pressure term, and by this recover the original Navier-Stokes equations we started with, provided the perforations' size is appropriately small. Lastly, to establish the strong convergence of the density, usually obtained by the so-called effective viscous flux identity, one needs different tools than in the barotropic case, again due to the low integrability $p(\rho) \in L^1(D)$. We overcome this drawback by considering Orlicz spaces (e.g.~$L \log L$) and the $\BMO$ space, which are rather non-standard in the framework of homogenization.

The technical heart of our analysis relies on a refined improved pressure estimate. Previous works on singular pressures of hard-sphere type -- notably the pioneering study by Feireisl and Zhang \cite{FeireislZhang2010} and the subsequent extension by Feireisl, Lu, and M\'alek \cite{FeireislLuMalek2016} -- typically rely on the equi-integrability of the pressure $p(\rho)$ by establishing a uniform bound in the Zygmund space $L \log L((0,T) \times D)$. In contrast, our homogenization result requires a more delicate estimate in the stronger Orlicz space $L(\log L)^{1+\kappa}((0,T) \times D)$ for $\kappa > 0$, as detailed in Appendix~\ref{sec:AppB}. This enhanced regularity is essential to compensate for the additional oscillations introduced by the perforated domain and to ensure the vanishing of the concentration terms during the limit passage, which is the main technical novelty of the present work. \\

\paragraph{\textbf{Notation.}} We use the standard notations for Lebesgue and Sobolev spaces, and denote them even for vector- or matrix-valued functions as in the scalar case, e.g., we use $L^p(D)$ instead of $L^p(D;\R^d)$. The Frobenius inner product of two matrices $A,B\in \R^{d\times d}$ is denoted by $A:B=\sum_{i,j=1}^d A_{ij} B_{ij}$. Moreover, we use the notation $a\lesssim b$ whenever there is a generic constant $C>0$ which is independent of $a$, $b$, and $\eps$ such that $a\leq C b$. We denote for a function $f$ with domain of definition $G \subset\R^d$ its zero extension by $\tilde{f}$, that is,
\begin{align*}
    \tilde{f}=f \text{ in } G,\qquad \tilde{f}=0 \text{ in } \R^d\setminus G,
\end{align*}
with the convention that $0 \cdot \log 0 := 0$ whenever it occurs. The space $L_0^q(G)$ is defined as all functions $f\in L^q(G)$ with $\int_G f \dd x = 0$. Lastly, we define the gradient of a vector by $(\nabla a)_{ij} = \del_i a_j = \del a_j / \del x_i$, and the divergence of a matrix column-wise by $(\div A)_i = \div(A e_i)$, where $e_i$ is the $i$-th canonical basis vector. Note that this is a non-standard convention, however, we believe it to be more compatible than the usual one since then, the sum of the velocity $\vu$ and the pressure gradient $\nabla p$ makes sense.

In addition to the standard Lebesgue spaces, our analysis relies on specific properties of Orlicz (Zygmund) and $\BMO$ spaces. A detailed discussion of these spaces, including the necessary embeddings and duality relations used throughout this work, is deferred to Appendix~\ref{sec:AppB}. \\

\paragraph{\textbf{Organization of the paper.}} In Section~\ref{sec2}, we introduce the model under consideration, weak solutions to this model, and formulate our main result. In Section~\ref{sec:Bds}, we give uniform bounds on the functions needed in the sequel. Section~\ref{sec:Conv} is devoted to show the convergence results using special cut-off functions. Finally, in Section~\ref{sec:strDens}, we show the strong convergence of the density to complete the proof of our main theorem, with special emphasis on the most nontrivial term involving the pressure. The treatment of this term in our analysis relies on the functional framework and results detailed in Appendix~\ref{sec:AppB}.


\section{The model, weak solutions, and the main result}\label{sec2}
To start, we introduce the perforated domain and the evolutionary compressible Navier-Stokes equations, and state our main result. We start with the description of the perforated domain and the equations governing the fluid's motion.
\subsection{The perforated domain and the Navier-Stokes equations}
Let $d=2,3$ and $D\subset \R^d$ be a bounded domain with smooth boundary, and for $\eps\in (0,1)$, let $K_\eps \subset D$ be a compact set. Moreover, let $\alpha \geq 1$, and set $a_\eps = \exp(-\eps^{-\alpha})$ if $d=2$, and $a_\eps = \eps^\alpha$ if $d=3$. We assume that there exists a family of balls $B_{a_\eps}(x_i(\eps))$, $i=1,...,N(\eps)$, such that
\begin{align*}
    \begin{split}
        K_\eps \subset \bigcup_{i=1}^{N(\eps)} B_{a_\eps}(x_i(\eps)),\\
        \dist(x_i(\eps), \del D) > \eps,\\
        \forall i\neq j: |x_i(\eps)-x_j(\eps)| > 2 \eps.
    \end{split}
\end{align*}
Note that this implies
\begin{align*}
    |K_\eps| \lesssim N(\eps) a_\eps^d \lesssim (a_\eps / \eps)^d.
\end{align*}
Finally, we define
\begin{align}
    D_\eps= D\setminus K_\eps. \label{defDeps}
\end{align}

For fixed $T>0$, we consider in $(0,T)\times D_\eps$ the evolutionary compressible Navier-Stokes equations
\begin{align}\label{NSE}
    \begin{cases}
        \del_t \rhoe + \div(\rhoe \vu_\eps)=0 & \text{in } (0,T)\times D_\eps,\\
        \del_t(\rhoe \vu_\eps) + \div(\rhoe \vu_\eps \otimes \vu_\eps) + \nabla p(\rhoe) = \div \ms(\nabla \vu_\eps) + \rhoe \vc f & \text{in } (0,T)\times D_\eps,\\
        \vu_\eps=0 & \text{on } (0,T)\times \del D_\eps,\\
        \rhoe(0,\cdot)=\rho_{\eps, 0},\ (\rhoe\vu_\eps)(0,\cdot)=\vc m_{\eps, 0} & \text{in } D_\eps.
    \end{cases}
\end{align}
Here, $\rhoe$ and $\vu_\eps$ denote the fluid's density and velocity, respectively, $\ms(\nabla \vu) = \mu (\nabla \vu + \nabla^T \vu - \frac2d \div \vu \mathbb I) + \mu_1 \div \vu \mathbb I$ is the stress tensor with shear and bulk viscosity coefficients $\mu>0, \mu_1 \geq 0$, and the force $\vc f\in L^2((0,T)\times D)$ as well as initial data $0 \leq \rho_{\eps, 0} \leq \ovr, \, \vc m_{\eps, 0} \in L^2(D)$ are given, where $\ovr > 0$ is a given constant. We remark that for simplicity, we chose $\vc f$ to be independent of $\eps$; the case where $\vc f_\eps \to \vc f$ strongly in $L^2((0,T) \times D)$ can be handled similarly.\\

The main part is the description of the pressure $p(\rhoe)$. As mentioned in the introduction, we consider a hard sphere pressure law such that there is a constant $\beta > 1 + \frac{d}{2}$ with
\begin{align}\label{gen-pr}
                p\in C^1([0,\ovr)), &&  p(0)=0, && p' >0 \text{ on }(0,\ovr), && \lim_{\varrho\rightarrow \ovr} p(\varrho)(\ovr - \rho)^\beta = c_\ovr > 0 . 
            \end{align}

The prototypical example the reader shall have in mind is a pressure of the form\footnote{Sometimes this kind of pressures is called \emph{singular congestion pressure}, see, e.g., \cite{BreschPerrinZatorska2014, DegondHua2013, DegondHuaNavoret2011}.}
\begin{align}\label{pressure}
    p(\rho) = \frac{\rho^\varsigma}{(\ovr-\rho)^\beta}, && \varsigma \geq 1, && \beta > 1 + \frac{d}{2}, && \ovr > 0,
\end{align}
see Figure~\ref{fig:p}.	Note that this \emph{a priori} guarantees $\rhoe \in L^\infty(0,T; L^\infty (D_\eps))$ with $0 \leq \rhoe < \ovr$.

\begin{figure}[h]\label{fig:p}
    \begin{tikzpicture}
            \draw[->, thick](0,0) -- (3,0) node[right] {$\rho$};
            \draw[->, thick](0,0) -- (0,6) node[left] {$p(\rho)$};
            \draw[domain=0:{2*(6 - sqrt(6))/5}, smooth, variable=\x, black] plot ({\x}, {
            \x^2 / (2 - \x)^2
            });
            \draw[dotted] (1.65,6.0) -- (1.65,0) node[below] {$\ovr$} ;
    \end{tikzpicture}
    \caption{Schematic pressure of the form \eqref{pressure}}
\end{figure}
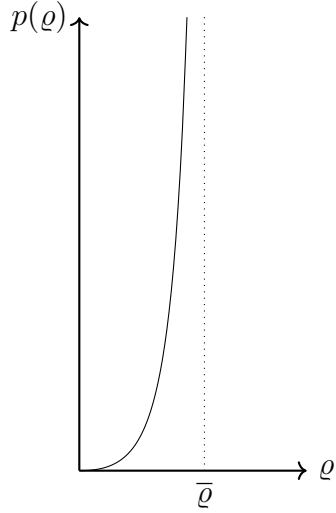


\subsection{Weak solutions and main result}
For further use, we introduce the concept of finite energy weak solutions.
    \begin{defin}\label{def1}
        Let $T>0$ be fixed and $Q \subset \R^d$ be a bounded Lipschitz domain, and let the initial data satisfy
        \begin{align*}
            \rho(0,\cdot)=\rho_0,\quad (\rho\vu)(0,\cdot)=\vc m_0,
        \end{align*}
        together with the compatibility conditions
        \begin{equation}\label{init}
            \begin{gathered}
                0\leq  \rho_0 \leq \ovr \text{ a.e.~in } Q, \quad \vc m_0=0 \text{ on } \{\rho_0=0\}, \quad \vc m_0 \in L^2(Q), \ \frac{|\vc m_0|^2}{\rho_0} \in L^1(Q), \\ 
                \text{ and } \int_{Q} P(\rho_0) \dd x<\infty ,
            \end{gathered}
        \end{equation}
        where $P(\rho)$ is the pressure potential defined by
            \begin{align}\label{press-pot}
                P(\rho) = \rho \int_{\ovr / 2}^\rho \frac{p(z)}{z^2} \dd z .
            \end{align}
            
        We call a duplet $(\rho,\vu)$ a \emph{renormalized finite energy weak solution} to system \eqref{NSE} if:
        \begin{itemize}[leftmargin=*]
            \item The solution belongs to the regularity class
            \begin{gather*}
                0\leq \rho \leq \ovr \text{ a.e. in } (0,T)\times Q,\; \rho \in C_{w}([0,T];L^{\gamma}(Q)) \text{ for any } \gamma \in [1, \infty), \\ 
                p(\rho) \in L^{1}((0,T)\times Q),\; \vu \in L^2(0,T;W_0^{1,2} (Q; \mathbb{R}^d)) , \\
                \rho \vu \in C_{w}([0,T];L^{2}(Q;\mathbb{R}^d)),\; \rho \vert \vu \vert^2 \in L^{\infty}(0,T;L^{1}(Q)).
            \end{gather*}
            \item We have
            \begin{equation}\label{renCE}
                \begin{split}
                    \del_t \rho + \div(\rho \vu) &= 0 \text{ in } \mathcal{D}'((0,T)\times Q),\\
                    \del_t \tvr + \div(\tvr \tvu) &= 0 \text{ in } \mathcal{D}'((0,T)\times \R^3), \\
                    \del_t b(\tvr)+\div(b(\tvr)\tvu)+(\tvr b'(\tvr)-b(\tvr))\div\tvu &= 0 \text{ in } \mathcal{D}'((0,T)\times \R^3),
                \end{split}
            \end{equation}
            for any $b\in C^1([0,\ovr))$ satisfying
            \begin{align*}
                \vert b(s) \vert^2 + \vert b'(s)\vert^2 \leq C(1 + p(s)) \text{ for some constant } C>0 \text{ and any } s\in [0,\ovr).
            \end{align*}

            \item For any $\phi\in C_c^\infty([0,T)\times Q; \R^3)$,
            \begin{align}\label{wkMom}
                \begin{split}
                    &\int_0^T\int_{Q} \rho \vu \cdot \del_t \phi \dd x \dd t + \int_0^T\int_{Q} \rho \vu\otimes \vu : \nabla \phi \dd x \dd t + \int_0^T\int_{Q} p(\rho) \div \phi \dd x \dd t \\
                    &- \int_0^T\int_{Q} \ms(\nabla \vu):\nabla \phi \dd x \dd t + \int_0^T\int_{Q} \rho \vc f \cdot \phi \dd x \dd t = -\int_{Q} \vc m_0 \cdot \phi(0,\cdot) \dd x .
                \end{split}
            \end{align}
            
            \item For almost any $\tau \in [0,T]$, the energy inequality holds:
            \begin{align}\label{EI}
                \begin{split}
                    &\int_{Q} \left[ \frac12 \rho |\vu|^2 + P(\rho) \right] (\tau, \cdot) \dd x + \int_0^\tau \int_{Q} \ms(\nabla\vu):\nabla\vu \dd x \dd t \\
                    &\quad \leq \int_{Q} \frac{|\vc m_0|^2}{2\rho_0} + P(\rho_0) \dd x + \int_0^\tau \int_{Q} \rho \vc f\cdot\vu \dd x \dd t .
                \end{split}
            \end{align}
        \end{itemize}
    \end{defin}

Regarding existence of weak solutions, we have the following
\begin{theorem}
    Let $D_\eps\subset \R^3$ be a bounded domain with Lipschitz boundary, $T>0$ be given. Let the pressure $p$ satisfy \eqref{gen-pr} and let the initial data satisfy \eqref{init}. Then, there exists a weak solution $(\rho,\vu)$ to system \eqref{NSE} in the sense of Definition~\ref{def1} with $Q=D_\eps$.
\end{theorem}
The existence of weak solutions for pressure laws satisfying \eqref{gen-pr} was first proved by Feireisl and Zhang \cite{FeireislZhang2010} for $\beta \geq 3$ in the case $d=3$, and was later extended by Feireisl, Lu, and M\'alek  \cite{FeireislLuMalek2016} to the range $\beta > \frac{5}{2}$ for $d=3$. While the corresponding results for $d=2$ do not appear to be stated explicitly in the literature, a careful inspection of the critical term estimates (see Section~\ref{sec:Bds}) suggests that, for general dimension $d$, the condition $\beta > 1 + \frac{d}{2}$ should be sufficient.

We are interested in the limiting behaviour of system \eqref{NSE} when $\eps \to 0$ for \emph{tiny} perforations. For such tiny holes, one should expect that these holes do not have a strong influence on the fluid, and the equations do not change their form in the limit. This is already known for barotropic pressures, see for instace \cite{DieningFeireislLu2017, NecasovaPan2022, NecasovaOschmann2023, OschmannPokorny2023, PokornySkrisovsky2021}. We will show that the same conclusion also holds true for hard sphere pressures. Therefore, our main result reads as follows:
\begin{theorem}\label{thm1}
    Let $d=2,3$, $\alpha > d$, $D \subset \R^d$ be a bounded Lipschitz domain, and $D_\eps$ be as in \eqref{defDeps}. Let $\vc f \in L^2((0,T) \times \R^d)$, $\beta > 1 + \frac{d}{2}$, and let $\{(\rhoe, \vu_\eps)\}_{\eps \in (0,1)}$ be a sequence of weak solutions to system \eqref{NSE} in the domain $D_\eps$ emanating from the initial data $(\rho_{\eps, 0}, \vc m_{\eps, 0})$ satisfying
    \begin{align*}
        \tvr_{\eps, 0} \to \rho_0 \text{ in } L^1(D), && \frac{|\tilde{\vc m}_{\eps, 0}|^2}{\tvr_{\eps, 0}} \to \frac{|\vc m_0|^2}{\rho_0} \text{ in } L^1(D).
    \end{align*}
    Then there is a (not relabeled) subsequence such that $\tvr_\eps \weak^\ast \rho$ weakly-($\ast$) in $L^\infty(0,T; L^\infty(D))$ and $\tvu_\eps \weak \vu$ weakly in $L^2(0,T; W_0^{1,2}(D))$, and the limit $(\rho, \vu)$ is a weak solution to system \eqref{NSE} in the sense of Definition~\ref{def1} with domain $Q=D$ emanating from the initial data $(\rho_0, \vc m_0)$.
\end{theorem}


\section{Uniform bounds}\label{sec:Bds}
In this section, we show uniform in $\eps$ bounds on the velocity, density, and pressure.

\subsection{Bounds obtained from the energy inequality}
\begin{lemma}
    Under the assumptions of Theorem~\ref{thm1}, we have
    \begin{align*}
        \begin{aligned}
            \|P(\rhoe)\|_{L^\infty(0,T; L^1(D_\eps))} + \|\sqrt{\rhoe}\vu_\eps\|_{L^\infty(0,T;L^2(D_\eps))} + \| \vu_\eps \|_{L^2(0,T; W_0^{1,2}( D_\eps))} \leq C,
        \end{aligned}
    \end{align*}
    for some constant $C>0$ which is independent of $\eps$. In particular, by virtue of the singular character of $p(\rhoe)$ in \eqref{gen-pr} and, in turn, of $P(\rhoe)$ as defined in \eqref{press-pot}, we also have
    \begin{align*}
        \|\rhoe\|_{L^\infty(0,T; L^\infty(D_\eps))} \leq C.
    \end{align*}
\end{lemma}
\begin{proof}
    By the energy inequality \eqref{EI} and the assumptions on the initial data \eqref{init}, we obtain
    \begin{align*}
        &  \int_{D_\eps} \left[ \frac{1}{2} \rhoe |\vu_\eps|^2 + P(\rhoe) \right](\tau, \cdot) \dd x + \int_0^\tau \int_{D_\eps} \ms(\nabla\vu_\eps):\nabla\vu_\eps \dd x \dd t\\
        &\leq C + \int_0^\tau\int_{D_\eps} \rhoe \vc f \cdot \vu_\eps \dd x \dd t.
    \end{align*}
    Using now H\"older's and Young's inequalities, we get for almost any $\tau\in [0,T]$ that
    \begin{align*}
        &\int_{D_\eps} \rhoe\vc f \cdot \vu_\eps \dd x \leq C \|\sqrt{\rhoe}\|_{L^2(D_\eps)} \|\sqrt{\rhoe} \vu_\eps\|_{L^2(D_\eps)} \leq C \|\rhoe\|_{L^1(D_\eps)} + \frac12 \|\rhoe |\vu_\eps|^2\|_{L^1(D_\eps)}\\
        &\leq C + \frac12 \|\rhoe |\vu_\eps|^2\|_{L^1(D_\eps)},
    \end{align*}
    where we used mass conservation $\|\rhoe(\tau, \cdot)\|_{L^1(D_\eps)} = \|\rho_{\eps, 0}\|_{L^1(D_\eps)} \leq C$. Thus, we end up with the inequality
    \begin{align*}
         \int_{D_\eps} \left[ \frac{1}{2} \rhoe |\vu_\eps|^2 + P(\rhoe) \right] (\tau, \cdot) \dd x + \int_0^\tau \int_{D_\eps} \ms(\nabla\vu_\eps):\nabla\vu_\eps \dd x \dd t \leq \int_0^\tau \int_{D_\eps} \frac{1}{2} \rhoe |\vu_\eps|^2 \dd x \dd t + C  ,
    \end{align*}
    yielding by virtue of Gr\"onwall's inequality
    \begin{align*}
        \sup_{t \in (0,T)} \int_{D_\eps} \Big[ \rhoe |\vu_\eps|^2 + P(\rhoe) \Big] (t, \cdot) \dd x + \int_0^T \int_{D_\eps} \ms(\nabla\vu_\eps):\nabla\vu_\eps \dd x \dd t \leq C  .
    \end{align*}
    Since we may extend both $\rhoe$ and $\vu_\eps$ by zero to $D$ without influencing the inequality, by virtue of the standard Korn and Poincar\'e inequalities in $D$ we conclude easily.
\end{proof}


\subsection{Improved pressure estimates}
The bounds for the pressure potential $P$ are not enough to conclude that $p(\rhoe)$ has a weak limit in $L^1((0,T) \times D)$, which is due to the less singular behaviour of $P$ versus $p$. To overcome this, we test the momentum equation against a function that ``recovers'' the pressure $p$. To this end, we introduce an inverse of the divergence that was constructed in \cite{DieningFeireislLu2017, HoeferKowalczykSchwarzacher2021} for $d=3$, and in \cite{NecasovaOschmann2023} for $d=2$ (see also \cite{LuSchwarzacher2018}):
  \begin{lemma}\label{lem:Bog}
Let $d=2,3$, $\alpha > d$, and $a_\eps = \exp(-\eps^{-\alpha})$ if $d=2$, $a_\eps = \eps^\alpha$ if $d=3$. For any $q \geq 1$, there is a bounded linear operator $\B_\eps : L_0^q(D_\eps) \to W_0^{1,q}(D_\eps)$ such that
\begin{align*}
\div \B_\eps(f)=f \text{ for any } f \in L_0^q(D_\eps), && \|\B_\eps \|_{L_0^q \to W_0^{1,q}}^q \lesssim 1+C(\eps, q),
\end{align*}
where:
\begin{itemize}
\item if $d=2$, then
\begin{align*}
    C(\eps, q) = \eps^{-\alpha} a_\eps^{2-q} \begin{cases}
        |\log (\eps^\frac{\alpha}{2} a_\eps^{-1})|^{-q} |(\eps^\frac{\alpha}{2} a_\eps^{-1})^{2-q} - 1| & \text{if } q \neq 2,\\
        |\log (\eps^\frac{\alpha}{2} a_\eps^{-1})|^{-1} & \text{if } q=2.
    \end{cases}
\end{align*}
In particular, $\B_\eps$ is a uniformly bounded operator from $L_0^q(D_\eps)$ to $W_0^{1,q}(D_\eps)$ for any $1 \leq q \leq 2$.
\item if $d=3$, then
\begin{align*}
    C(\eps, q) = \eps^{(3-q)\alpha - 3}.
\end{align*}
In particular, $\B_\eps$ is a uniformly bounded operator from $L_0^q(D_\eps)$ to $W_0^{1,q}(D_\eps)$ for any $q \geq 1$ with $(3-q)\alpha - 3 \geq 0$. Since $\alpha > 3$ if $d=3$, this holds at least for any $1 \leq q \leq 2$.
\end{itemize}

Moreover, this operator can be extended to $\B_\eps : \dot{W} := \{g \in [W^{1,2}(D_\eps)]' : \langle g, 1 \rangle = 0 \} \to L^2(D_\eps)$ such that for any $\vc f \in L^2(D_\eps)$ with $\langle \div \vc f , 1 \rangle = 0$,
\begin{align*}
    \B_\eps (\div \vc f) = \vc f, && \|\B_\eps\|_{\dot{W} \to L^2} \lesssim 1.
\end{align*}
\end{lemma}

Having the operator $\B_\eps$ at hand, we will show:
\begin{lemma}\label{imprPress}
    For $\beta > 1 + \frac{d}{2}$, it holds that
    \begin{align}\label{pre-equi}
        \rhoe p(\rhoe) \in L^1((0,T) \times D_\eps), && p(\rhoe) \log(1 - \rhoe/\ovr) \in L^1((0,T) \times D_\eps),
    \end{align}
    where both functions are uniformly bounded with respect to $\eps$. In particular, we have that $p(\rhoe)$ is uniformly bounded in $L^1((0,T) \times D_\eps)$, and the sequence $\{p(\rhoe)\}_{\eps \in (0,1)}$ is equi-integrable.
    
    Moreover, for any $\kappa \in (0, \infty)$, we have
    \begin{align*}
        p(\rhoe) \lr{\log(1 - \rhoe/\ovr)}^{1+\kappa} \in L^1((0,T) \times D_\eps)
    \end{align*} 
    uniformly with respect to $\eps$. As a consequence, the pressure sequence is uniformly bounded in the space-time Orlicz space
    \begin{align*}
         p(\rhoe) \in L(\log L)^{1+\kappa} ((0,T)\times D_\eps),
    \end{align*}
    and thus, the sequence $\{p(\tvr_\eps) \log p(\tvr_\eps)\}_{\eps \in (0,1)}$ is equi-integrable as well.
\end{lemma}
\begin{proof}
The equi-integrability and boundedness of $p(\rhoe)$ follow from \eqref{pre-equi} since for any $\kappa \geq 0$,
				\begin{align}\label{eqiint-pr}
				p(\rhoe) + p(\rhoe) |\log p(\rhoe)|^{1+\kappa} &\leq C \left(1+\rhoe p(\rhoe) + p(\rhoe) |\log(1 - \rhoe / \ovr) |^{1+\kappa} \right).
			\end{align}

            Indeed, clearly
            \begin{align*}
                p(\rhoe) = p(\rhoe) \mathbf{1}_{\rhoe < \ovr/2} + p(\rhoe) \mathbf{1}_{\rhoe \geq \ovr/2} \leq C (1 + \rhoe p(\rhoe)),
            \end{align*}
            and from the assumptions on the pressure in \eqref{gen-pr}, we can find some $\delta \in (0, \ovr)$ such that
			\begin{align*}
				\frac{c_\ovr }{2} \leq p(\rho)(\ovr-\rho)^\beta \leq \frac{3 c_\ovr}{2} \text{ for } \ovr-\delta < \rho <\ovr ,
			\end{align*} 
			where the choice of $\delta$ only depends on $\beta$ and the structure of the pressure $p$. 
			Thus, for any $\kappa \geq 0$, we have 
            \begin{align*}
    p(\rho_\eps) | \log p(\rho_\eps) |^{1+\kappa} 
    &\leq p(\rho_\eps) |\log p(\rho_\eps)|^{1+\kappa} \mathbf{1}_{0 \leq \rho_\eps \leq \ovr-\delta} + p(\rho_\eps) |\log p(\rho_\eps) |^{1+\kappa} \mathbf{1}_{\ovr-\delta < \rho_\eps < \ovr} \\
    &\leq C + p(\rho_\eps) \left| \log \left[ \frac{p(\rho_\eps)(1-\rho_\eps/\ovr)^\beta}{(1-\rho_\eps/\ovr)^\beta} \right] \right|^{1+\kappa} \mathbf{1}_{\ovr-\delta < \rho_\eps < \ovr} \\
    &= C + p(\rho_\eps) \left| \log[p(\rho_\eps)(1-\rho_\eps/\ovr)^\beta] - \beta \log (1-\rho_\eps/\ovr) \right|^{1+\kappa} \mathbf{1}_{\ovr-\delta < \rho_\eps < \ovr} \\
    &\leq C \lr{1 + p(\rho_\eps) \left| \log[p(\rho_\eps)(1-\rho_\eps/\ovr)^\beta] \right|^{1+\kappa} \mathbf{1}_{\ovr-\delta < \rho_\eps < \ovr} + p(\rho_\eps) \left| \log(1-\rho_\eps/\ovr) \right|^{1+\kappa} } \\
    &\leq C \left( 1 + p(\rho_\eps) + p(\rho_\eps) \left| \log(1 - \rho_\eps/\ovr) \right|^{1+\kappa} \right) \\
    &\leq C \left( 1 + \rho_\eps p(\rho_\eps) + p(\rho_\eps) \left| \log(1 - \rho_\eps/\ovr) \right|^{1+\kappa} \right),
\end{align*}
where the constant $C>0$ just depends on $\ovr, \delta, \beta$, and $\kappa$. The equi-integrability now follows by a De la Vall\'ee-Poussin argument since $s \mapsto s (\log s)^{1+\kappa}$ is superlinear when $s \to \infty$.\\

To show \eqref{pre-equi}, let us start with a more general observation. We test the momentum equation with $ \phi = \psi(t) \B_\eps \left[  b(\rhoe)- \frac{1}{| D_\eps|  } \int_{D_\eps} b(\rhoe ) \dd y \right] $ for some $\psi \in C^\infty_c(0,T)$. After a straightforward calculation, we find\footnote{For details, see, e.g., \cite{FeireislNovotny2009singlim}.}
\begin{align}\label{bog1}
	\int_0^T \psi \int_{D_\eps} p(\rhoe) \left[ b(\rhoe) - \frac{1}{|D_\eps|} \int_{D_\eps} b(\rhoe) \dd y \right] \dd x \dd t = \sum_{i=1}^{6} \mathcal{A}_i,
\end{align} 
where
\begin{align*}
    &  \mathcal{A}_1 = 	-\int_0^T \partial_t \psi   \int_{D_\eps} \rhoe \vu_\eps \cdot \B_\eps \left[  b(\rhoe)- \frac{1}{| D_\eps|  } \int_{D_\eps} b(\rhoe ) \dd x \right]  \dd x \dd t ,\\
    &  \mathcal{A}_2 = 	\int_0^T \psi \int_{D_\eps} \B_\eps \left[ \div ( b(\rhoe) \vu_\eps) \right]\dd x \dd t ,\\
    &  \mathcal{A}_3 = 	\int_0^T \psi \int_{D_\eps} \rhoe \vu_\eps \cdot  \B_\eps \left[ ( \rhoe b^\prime (\rhoe )- b(\rhoe))  \div \vu_\eps - \frac{1}{|D_\eps|} \int_{D_\eps} ( \rhoe b^\prime (\rhoe )- b(\rhoe))  \div \vu_\eps \dd y \right] \dd x \dd t ,\\
    &  \mathcal{A}_4 = 	\int_0^T \psi \int_{D_\eps} \ms(\nabla \vu_\eps )\; : \; \nabla \B_\eps \left[  b(\rhoe)- \frac{1}{| D_\eps|  } \int_{D_\eps} b(\rhoe ) \dd y \right]  \dd x \dd t ,\\
    &  \mathcal{A}_5 = - 	\int_0^T \psi \int_{D_\eps} \rhoe \vu_\eps \otimes \vu_\eps : \nabla \B_\eps \left[  b(\rhoe)- \frac{1}{| D_\eps|  } \int_{D_\eps} b(\rhoe ) \dd y \right] \dd x \dd t ,\\
    &  \mathcal{A}_6= -	\int_0^T \psi \int_{D_\eps} \rhoe \mathbf{f} \cdot \B_\eps \left[  b(\rhoe)- \frac{1}{| D_\eps|  } \int_{D_\eps} b(\rhoe ) \dd y \right] \dd x \dd t .
\end{align*}

The claim \eqref{pre-equi} is then proved in two steps.\\

\paragraph{\textbf{Step 1:}} Consider $b(\rhoe)=\rhoe$. In this case, first we observe that $\mathcal{A}_3=0$. Moreover, we denote 
\[  \B_\eps^{\ast} =  \B_\eps \left[  \rhoe - \frac{1}{| D_\eps| }  \int_{D_\eps} \rhoe \dd y \right] , \]
and notice that from Lemma~\ref{lem:Bog} and the uniform bound on $\rhoe \in L^\infty(0,T; L^\infty(D_\eps))$ we have
\begin{align*}
\|\B_\eps^\ast\|_{L^\infty(0,T; W^{1,2}(D_\eps))} \lesssim \|\rhoe\|_{L^\infty(0,T; L^2(D_\eps))} \lesssim 1.
\end{align*}
We estimate each of the $\mathcal A_i$ separately. To ease the presentation, we will use the notation $L^p L^q$ instead of $L^p(0,T; L^q(D_\eps))$.
\begin{itemize}[leftmargin=*]
\item For $\mathcal A_1$:
\begin{align*}
    |\mathcal A_1| \lesssim \|\rhoe\|_{L^\infty L^\infty} \|\vu_\eps\|_{L^2 L^2} \|\B_\eps^*\|_{L^2 L^2} \lesssim 1.
\end{align*}
\item For $\mathcal A_2$: since $\B_\eps : \dot{W} \to L^2$ is a uniformly bounded operator, we find
\begin{align*}
    |\mathcal A_2| \lesssim \|\B_\eps [\div (\rhoe \vu_\eps)]\|_{L^2 L^2} \lesssim \|\rhoe \vu_\eps\|_{L^2 L^2} \lesssim 1 .
\end{align*}
The fact that $\rhoe \vu_\eps \in \dot{W}$ follows from continuity equation and uniform bounds on the functions (see \cite{FeireislNovotny2009singlim} if needed).
\item For $ \mathcal{A}_4 $:
\begin{align*}
    |\mathcal{A}_4| \lesssim \Vert \ms(\nabla \vu_\eps ) \Vert_{L^2 L^2}  \Vert \nabla \B_\eps^{\ast}   \Vert_{L^2 L^2} \lesssim 1 .
\end{align*}
\item For $ \mathcal{A}_5 $: thanks to Sobolev embedding $W^{1,2} \hookrightarrow L^4$, we get
\begin{align*}
    |\mathcal{A}_5| \lesssim \Vert \rhoe \Vert_{L^\infty L^\infty} \Vert \vu_\eps \otimes \vu_\eps  \Vert_{L^1 L^2}  \Vert \nabla \B_\eps^{\ast}   \Vert_{L^\infty L^2} \lesssim 1 .
\end{align*}
\item For $\mathcal A_6$:
\begin{align*}
    |\mathcal A_6| \lesssim \|\rhoe\|_{L^\infty L^\infty} \|\vc f\|_{L^2 L^2} \|\B_\eps^*\|_{L^2 L^2} \lesssim 1.
\end{align*}

In total, we now have
\begin{align*}
    \int_0^T \psi \int_{D_\eps} p(\rhoe) \left[ \rhoe - \frac{1}{|D_\eps|} \int_{D_\eps} \rhoe \dd y \right] \dd x \dd t \lesssim 1. 
\end{align*}
To get rid of the mean value, we deduce from $\rho_{\eps, 0} \leq \ovr$ and $\int_{D_\eps} P(\rho_{\eps, 0}) \dd x < \infty$ that\footnote{Note that if $m=\ovr$, then $\int_{D_\eps} (\rho_{\eps, 0} - \ovr) \dd x = 0$, leading to $\rho_{\eps, 0} = \ovr$ a.e., contradiction.}
\begin{align*}
    \frac{1}{|D_\eps|} \int_{D_\eps} \rhoe (t) \dd x = \frac{1}{|D_\eps|} \int_{D_\eps} \rho_{\eps, 0} \dd x = m < \ovr.
\end{align*}
Then we can argue as in \cite{FeireislZhang2010} that
\begin{align*}
    &\int_0^T \psi \int_{D_\eps} p(\rhoe) \left[ \rhoe - \frac{1}{|D_\eps|} \int_{D_\eps} \rhoe \dd y \right] \dd x \dd t = J_1 + J_2,\\
    J_1 &:= \int_0^T \psi \int_{\rhoe < (m + \ovr)/2} p(\rhoe) \left[ \rhoe - \frac{1}{|D_\eps|} \int_{D_\eps} \rhoe \dd y \right] \dd x \dd t , \\
    J_2 &:= \int_0^T \psi \int_{\rhoe \geq (m+\ovr)/2} p(\rhoe) \left[ \rhoe - \frac{1}{|D_\eps|} \int_{D_\eps} \rhoe \dd y \right] \dd x \dd t \\
    &\geq \frac{\ovr - m}{2} \int_0^T \psi \int_{\rhoe \geq (m+\ovr)/2} p(\rhoe) \dd x,
\end{align*}
and $J_1$ is bounded since $D_\eps$ is a bounded domain, leading to a uniform bound for $J_2$ and in turn for $p(\rhoe) \in L^1 L^1$. This yields
\begin{align*}
    \int_0^T \psi \left[ \int_{D_\eps} p(\rhoe) \dd x \right] \left[ \frac{1}{|D_\eps|} \int_{D_\eps} \rhoe \dd y \right] \dd t \lesssim 1
\end{align*}
such that also $\rhoe p(\rhoe) \in L^1 L^1$ is uniformly bounded.\\
\end{itemize}
		
\paragraph{\textbf{Step 2:}} Consider 
\begin{align}\label{b_log}
    b(\rhoe) = \begin{cases}
        0 & \text{if } 0 \leq \rhoe \leq \frac{\ovr}{4},\\
        \text{smooth} & \text{if } \frac{\ovr}{4} \leq \rhoe \leq \frac{\ovr}{2}, \\
        -\log(1 - \rhoe / \ovr) & \text{if } \frac{\ovr}{2} \leq \rhoe <\ovr,
    \end{cases}
\end{align}
    where ``smooth'' means a connecting function (e.g.~a spline) such that $b(\rhoe) \in C^1([0,\ovr))$. We additionally introduce the notation $\eta(\rho)= \log (1 - \rho / \ovr) $. Note that then always $|b(\rhoe)| \leq C |\eta(\rhoe)|$ as well as $|b'(\rhoe)| \leq C |\eta'(\rhoe)|$, where $C>0$ just depends on $\ovr$ and the form of the spline.\\

First, we list the following observations already given in \cite{FeireislLuMalek2016}: 
\begin{itemize}[leftmargin=*]
    \item From the properties of $p$ in \eqref{gen-pr} and the definition of $P$ in \eqref{press-pot}, there exist positive constants $c_1$ and $c_2$ such that
    \begin{align*}
       P(\rho) \geq \frac{c_1}{(\ovr- \rho)^{\beta-1} } - c_2,
    \end{align*}
implying, in particular,
\begin{align*}
    P(\rho) \geq c_1 | \eta(\rho) |^q -c_2 \text{ for any } 1\leq q <\infty ,
\end{align*}
such that $\eta(\rhoe)$ is uniformly bounded in $L^\infty L^q$ for any $1 \leq q <\infty$.

\item Similarly, there exist positive constants $c_3$ and $c_4$ such that
\begin{align*}
        | p(\rho) | \geq  c_3 |\eta^\prime (\rho)|^\beta  -    c_4,
\end{align*}
giving that $\eta^\prime (\rhoe)$ is uniformly bounded in $L^\beta L^\beta$ since $p(\rhoe)$ is uniformly bounded in $L^1 L^1$ by the previous step.
\end{itemize} 

Let us now highlight why the condition $\beta>1 + \frac{d}{2}$ plays an important role. To this end, we focus on the term $\mathcal{A}_3$: 

\begin{align*}
    \mathcal{A}_3& = 	\int_0^T \psi \int_{D_\eps} \rhoe \vu_\eps \cdot  \B_\eps \left[ ( \rhoe b^\prime (\rhoe )- b(\rhoe))  \div \vu_\eps - \frac{1}{|D_\eps|} \int_{D_\eps} ( \rhoe b^\prime (\rhoe )- b(\rhoe))  \div \vu_\eps \dd y \right] \dd x \dd t \\ 
    &= 	\int_0^T \psi \int_{D_\eps} \rhoe \vu_\eps \cdot  \B_\eps \left[ \rhoe b^\prime (\rhoe )  \div \vu_\eps - \frac{1}{|D_\eps|} \int_{D_\eps}  \rhoe b^\prime (\rhoe )  \div \vu_\eps \dd y \right] \dd x \dd t \\ 
    &\quad -	\int_0^T \psi \int_{D_\eps} \rhoe \vu_\eps \cdot  \B_\eps \left[  b(\rhoe)  \div \vu_\eps - \frac{1}{|D_\eps|} \int_{D_\eps}  b(\rhoe)  \div \vu_\eps \dd y \right] \dd x \dd t \\ 
    &=: \mathcal{A}_{3,1}+ \mathcal{A}_{3,2}.
\end{align*}

For $\mathcal A_{3,2}$, thanks to $|b(\rhoe)| \leq C |\eta(\rhoe)|$ and the embeddings $W^{1,1} \hookrightarrow L^\frac{d}{d-1}$ and $W^{1,2} \hookrightarrow L^d$, we find
\begin{align*}
    |\mathcal A_{3,2}| &\lesssim \|\rhoe\|_{L^\infty L^\infty} \|\vu_\eps\|_{L^2 L^d} \left\Vert \B_\eps \left[  b(\rhoe)  \div \vu_\eps - \frac{1}{|D_\eps|} \int_{D_\eps} b(\rhoe)  \div \vu_\eps \dd y \right]  \right\Vert_{L^2 L^{\frac{d}{d-1}} } \\
    & \lesssim 	\left\Vert \B_\eps \left[  b(\rhoe)  \div \vu_\eps - \frac{1}{|D_\eps|} \int_{D_\eps}  b(\rhoe)  \div \vu_\eps \dd y \right]  \right\Vert_{L^2 W^{1,1} } \\ 
    & \lesssim \|\eta(\rhoe) \div \vu_\eps\|_{L^2 L^1} \lesssim \|\eta(\rhoe)\|_{L^\infty L^2} \lesssim 1 + \|P(\rhoe)\|_{L^\infty L^1} \lesssim 1.
\end{align*}

The most difficult term is $\mathcal{A}_{3,1}$. From the bounds of $\eta'(\rhoe)$ and $\div \vu_\eps$, we have
\begin{align*}
    \rhoe \eta^\prime (\rhoe)\div  \vu_{\eps} \text{ bounded in } L^{\frac{2\beta}{2+\beta}} L^{\frac{2\beta}{2+\beta}} 
\end{align*}
such that
\begin{align*}
    \B_\eps \left[\rhoe b'(\rhoe) \div \vu_\eps - \frac{1}{|D_\eps|} \int_{D_\eps} \rhoe b'(\rhoe) \div \vu_\eps \dd y \right] \text{ bounded in } L^\frac{2\beta}{2+\beta} W_0^{1, \frac{2\beta}{2+\beta}},
\end{align*}
and by Sobolev embedding that
\begin{align*}
    \B_\eps \left[\rhoe b'(\rhoe) \div \vu_\eps - \frac{1}{|D_\eps|} \int_{D_\eps} \rhoe b'(\rhoe) \div \vu_\eps \dd y \right] \text{ bounded in } L^\frac{2\beta}{2+\beta} L^\frac{2d\beta}{2d + \beta(d-2)} .
\end{align*}

Moreover, we have by Sobolev embedding $W_0^{1,2}(D_\eps) \hookrightarrow L^6(D_\eps)$ and $\rhoe \in L^\infty L^\infty$ that\footnote{If $d=2$, one might take arbitrary $L^q$ instead of $q=6$, however, this gives a bound for $\rhoe \vu_\eps$ in $L^r L^\frac{2rq}{4+q(r-2)}$ and thus changes the eventual inequality for $\beta$ to $\beta(q-1) \geq 2(q-1)$, giving no additional information.}
\begin{align*}
    \rhoe \vu_\eps \text{ is uniformly bounded in } L^\infty L^2 \cap L^2 L^6,
\end{align*}
which by interpolation gives for any $r \in [2,\infty]$ that
\begin{align*}
    \rhoe \vu_\eps \text{ is uniformly bounded in } L^r L^\frac{6r}{3r-4}.
\end{align*}
Hence, we may estimate in 3D that
\begin{align*}
    |\mathcal A_{3,1}| &\lesssim \|\rhoe \vu_\eps\|_{L^\frac{2\beta}{\beta-2} L^\frac{6\beta}{5\beta-6}} \left\Vert \B_\eps \left[\rhoe b'(\rhoe) \div \vu_\eps - \frac{1}{|D_\eps|} \int_{D_\eps} \rhoe b'(\rhoe) \div \vu_\eps \dd y \right] \right\Vert_{L^\frac{2\beta}{2+\beta} L^\frac{6\beta}{6+\beta}} \\
    &\lesssim \|\rhoe \eta'(\rhoe) \div \vu_\eps\|_{L^\frac{2\beta}{2+\beta} L^\frac{2\beta}{2+\beta}} \lesssim 1,
\end{align*}
provided $2\beta/(\beta-2) \geq 2$ which is always true, and
\begin{align*}
    \frac{6\beta}{5\beta-6} \leq \frac{6 \cdot 2\beta/(\beta-2)}{3 \cdot 2\beta/(\beta-2)-4} \Leftrightarrow \beta \geq \frac52.
\end{align*}

If $d=2$, then we estimate
\begin{align*}
    |\mathcal A_{3,1}| &\lesssim \|\rhoe \vu_\eps\|_{L^\frac{2\beta}{\beta-2} L^\frac{\beta}{\beta-1}} \left\Vert \B_\eps \left[\rhoe b'(\rhoe) \div \vu_\eps - \frac{1}{|D_\eps|} \int_{D_\eps} \rhoe b'(\rhoe) \div \vu_\eps \dd y \right] \right\Vert_{L^\frac{2\beta}{2+\beta} L^\beta} \\
    &\lesssim \|\rhoe \eta'(\rhoe) \div \vu_\eps\|_{L^\frac{2\beta}{2+\beta} L^\frac{2\beta}{2+\beta}} \lesssim 1,
\end{align*}
provided
\begin{align*}
    \frac{\beta}{\beta-1} \leq \frac{6 \cdot 2\beta/(\beta-2)}{3 \cdot 2\beta/(\beta-2) - 4} \Leftrightarrow \beta \geq 2.
\end{align*}
In turn, we conclude that $\mathcal A_{3,1}$ is uniformly bounded provided $\beta \geq 1+ \frac{d}{2}$.\\

Next, due to $\eta(\rhoe) \in L^\infty L^q$ uniformly for any $1 \leq q < \infty$, the terms $\mathcal A_i$, $i=1,4,5,6$, are uniformly bounded the same way as in Step~1. For $\mathcal A_2$, we notice that $\eta(\rhoe) \vc u_\eps \in L^2 L^2$ with $\eta(\rhoe) \vc u_\eps \cdot \vc n |_{\partial D_\eps} = 0$ in the sense of normal traces such that
\begin{align*}
    |\mathcal A_2| \lesssim \|\eta(\rhoe) \vc u_\eps\|_{L^2 L^2} \lesssim \|\eta(\rhoe)\|_{L^\infty L^3} \|\vc u_\eps\|_{L^2 L^6} \lesssim 1,
\end{align*}
so if $\beta \geq 1+ \frac{d}{2}$, we find uniformly in $\eps$ that $$\int_0^T \int_{D_\eps}  p(\rhoe) \left[ b(\rhoe) - \frac{1}{|D_\eps|} \int_{D_\eps} b(\rhoe) \dd y \right] \dd x \dd t \lesssim 1,$$ where $b(\rhoe)$ is as in \eqref{b_log}. Now
\begin{align*}
    \int_0^T \int_{D_\eps} p(\rhoe) \left[ \frac{1}{|D_\eps|} \int_{D_\eps} b(\rhoe) \dd y \right] \dd x \dd t &\lesssim \int_0^T \|p(\rhoe)\|_{L^1(D_\eps)} \|b(\rhoe)\|_{L^1(D_\eps)} \dd t \\
    &\lesssim \|p(\rhoe)\|_{L^1 L^1} \|\eta(\rhoe)\|_{L^\infty L^1} \lesssim 1
\end{align*}
thanks to the uniform bound on $p(\rhoe) \in L^1 L^1$ and the fact that $\eta(\rhoe)$ is uniformly bounded in $L^\infty L^q$ for any $q \geq 1$. Thus, we find
\begin{align*}
    p(\rhoe) b(\rhoe) \text{ uniformly bounded in } L^1((0,T) \times D_\eps).
\end{align*}
This implies by the fact that $|\eta(\rhoe)| = |\eta(\rhoe)| \mathbf{1}_{\rhoe \leq \ovr/2} + |\eta(\rhoe)| \mathbf{1}_{\rhoe \geq \ovr/2} \leq C(\ovr) + |b(\rhoe)|$ and the uniform bound on $p(\rhoe)$ in $L^1 L^1$ that also $p(\rhoe) \log(1 - \rhoe/\ovr) \in L^1((0,T) \times D_\eps)$ is uniformly bounded, provided $\beta \geq 1 + \frac{d}{2}$. This shows \eqref{pre-equi}.\\

\paragraph{\textbf{The $\kappa$-estimate:}} A crucial feature to show strong convergence of the density later on will be the equi-integrability of the sequence $\{p(\rhoe) \log p(\rhoe)\}_{\eps \in (0,1)}$. To ensure this, we seek a bound strictly stronger than $L \log L((0,T) \times D_\eps)$ for this sequence, which will require the stronger assumption $\beta > 1 + \frac{d}{2}$. To this end, we fix a parameter $\kappa \in (0, \infty)$ and define the renormalization function $b_\kappa(\rhoe)$ as
\begin{align*}
	b_{\kappa}(\rhoe) = \begin{cases}
		0 & \text{if } 0 \leq \rhoe \leq \frac{\ovr}{4},\\
		\text{smooth} & \text{if } \frac{\ovr}{4} \leq \rhoe \leq \frac{\ovr}{2}, \\
		[-\log(1 - \rhoe / \ovr)]^{1+\kappa} & \text{if } \frac{\ovr}{2} \leq \rhoe <\ovr,
	\end{cases}
\end{align*}
where ``smooth'' denotes a spline such that $b_\kappa(\rho_\eps) \in C^1([0,\ovr))$. We also denote $\eta_\kappa(\rho) = [-\log(1 - \rho / \ovr)]^{1+\kappa}$. By construction, we have $|b_\kappa(\rho)| \leq C|\eta_\kappa(\rho)|$ and $|b_\kappa'(\rho)| \leq C|\eta_\kappa'(\rho)|$, where
\begin{align*}
	\eta_\kappa'(\rho) = -(1+\kappa) [-\log(1 - \rho / \ovr)]^\kappa \frac{\ovr}{\ovr - \rho}.
\end{align*}
The same as in Step~2, we have from $P(\rho) \gtrsim (\ovr - \rho)^{-(\beta-1)} - 1$ that for any $1 \leq q < \infty$, $P(\rho) \geq c_1 |\eta_\kappa(\rho)|^q - c_2$, implying that $\{\eta_\kappa(\rhoe)\}_{\eps \in (0,1)}$ is uniformly bounded in $L^\infty L^q$ for all finite $q$. However, unlike in the case $\kappa=0$, we will see that the term $\eta_\kappa '(\rho)$ is just bounded in $L^{\beta - \delta} L^{\beta-\delta}$ for any $\delta \in (0, \beta-1]$. Indeed, given $p(\rho) \gtrsim (\ovr - \rho)^{-\beta} - 1$, we find
\begin{align*}
	|\eta_\kappa '(\rho)|^q \lesssim 1 + p(\rho) \Leftrightarrow [-\log(1- \rho / \ovr)]^{q\kappa} \leq C(\kappa, q) (\ovr - \rho)^{-(\beta-q)} \Leftrightarrow q \lneq \beta,
\end{align*}
where strict inequality comes from $\kappa > 0$. Note that $C(\kappa, q) \to \infty$ as $\kappa \to \infty$, but actually no bound on $\kappa$ is needed here. Moreover, since we assumed $\beta > 1 + \frac{d}{2}$, we always can choose $q=2$ such that $b_\kappa(\rho)$ is a proper function in the renormalized continuity equation \eqref{renCE}. This ensures that the source term $(\rho b_\kappa'(\rhoe) - b_\kappa(\rhoe))\div \vu_\eps$ is sufficiently integrable to be tested against the operator $\B_\eps$. \\

Now, we recall \eqref{bog1}, where we test the momentum equation with $$\phi = \psi(t) \B_\eps \left[  b_\kappa(\rhoe)  - \frac{1}{|D_\eps|} \int_{D_\eps}  b_\kappa(\rhoe) \dd y \right] .$$ Thanks to the bounds on $\eta_\kappa(\rhoe)$, which give accordingly the same bounds on $b_\kappa(\rhoe)$, we can follow the procedure of Step~1 for each of the integrals $\mathcal{A}_i$, $i = 1,2,4,5,6$, leading to $|\mathcal{A}_i| \lesssim 1$. The only crucial term here is, like in Step~2, $\mathcal{A}_3$. Recall this critical term as
\begin{align*}
    \mathcal{A}_3& = 	\int_0^T \psi \int_{D_\eps} \rhoe \vu_\eps \cdot  \B_\eps \left[ ( \rhoe b_{\kappa}^\prime (\rhoe )- b_{\kappa}(\rhoe))  \div \vu_\eps - \frac{1}{|D_\eps|} \int_{D_\eps} ( \rhoe b_{\kappa}^\prime (\rhoe )- b_{\kappa}(\rhoe))  \div \vu_\eps \dd y \right] \dd x \dd t \\ 
    &= 	\int_0^T \psi \int_{D_\eps} \rhoe \vu_\eps \cdot  \B_\eps \left[ \rhoe b_{\kappa}^\prime (\rhoe )  \div \vu_\eps - \frac{1}{|D_\eps|} \int_{D_\eps}  \rhoe b_{\kappa}^\prime (\rhoe )  \div \vu_\eps \dd y \right] \dd x \dd t \\ 
    &\quad -	\int_0^T \psi \int_{D_\eps} \rhoe \vu_\eps \cdot  \B_\eps \left[  b_{\kappa}(\rhoe)  \div \vu_\eps - \frac{1}{|D_\eps|} \int_{D_\eps}  b_{\kappa}(\rhoe)  \div \vu_\eps \dd y \right] \dd x \dd t \\ 
    &=: \mathcal{A}_{3,1}+ \mathcal{A}_{3,2}.
\end{align*}
The term $\mathcal{A}_{3,2}$ is handled the very same way as in Step~2, so we will just focus on $\mathcal{A}_{3,1}$ involving the singular derivative $b_\kappa'(\rhoe) \approx (\ovr-\rhoe)^{-1} [-\log(1 - \rhoe / \ovr)]^\kappa$. Recalling $(\ovr-\rhoe)^{-1} \in L^\beta L^\beta$, the logarithmic factor in $b_\kappa'(\rhoe)$ induces a  loss of integrability $\delta> 0$, such that $\rhoe b_\kappa'(\rhoe) \in L^{\beta-\delta} L^{\beta-\delta}$ for any $\delta \in (0, \beta-1]$. Note that this loss is indeed independent of $\kappa$, however, the norm of $\rhoe b_\kappa '(\rhoe)$ in $L^{\beta- \delta} L^{\beta - \delta}$ will of course depend on both parameters $\kappa$ and $\delta$. By H\"{o}lder's inequality with $\text{div } \vu_\eps \in L^2 L^2$, the source term satisfies
\begin{align*}
    \rhoe b_\kappa'(\rhoe) \text{div } \vu_\eps \in L^{\frac{2(\beta-\delta)}{2+\beta-\delta}} L^{\frac{2(\beta-\delta)}{2+\beta-\delta}} \text{ uniformly in } \eps.
\end{align*}
Now, we can repeat the steps for $\mathcal{A}_{3,1}$ in Step~2 upon replacing each $\beta$ by $\beta - \delta$ there, leading especially to
\begin{align*}
    &\lr{\frac{2d(\beta-\delta)}{2d+(\beta-\delta)(d-2)}}' = \frac{2d(\beta - \delta)}{(\beta-\delta)(d+2)-2d} \leq \frac{6 \cdot 2(\beta-\delta)/(\beta-\delta-2)}{3 \cdot 2(\beta-\delta)/(\beta-\delta-2)-4} \\
    &\Leftrightarrow \beta - \delta \geq 1 + \frac{d}{2} .
\end{align*}
Hence, we can choose $\delta \in (0, \beta - 1 - \frac{d}{2})$ to ensure that $|\mathcal{A}_{3,1}| \lesssim 1$, where the implicit constant depends on the chosen values of $\delta \in (0, \beta - 1 - \frac{d}{2})$ and $\kappa \in (0, \infty)$. \\

In total, we find that $p(\rhoe) b_\kappa(\rhoe)$ is uniformly bounded in $L^1 L^1$, which implies $p(\rhoe) |\log p(\rhoe)|^{1+\kappa} \in L^1 L^1$ uniformly in $\eps$, establishing the equi-integrability of the sequence $\{p(\tvr_\eps) \log p(\tvr_\eps)\}_{\eps \in (0,1)}$ by choosing $\Psi(s) = s [\log s]^\kappa$ in the De la Vall\'ee-Poussin criterion, and the fact that for $t \geq e^2$, we have $\Psi(t \log t) \leq 2^\kappa t [\log t]^{1+\kappa}$.

\end{proof}

	
\section{Convergence}\label{sec:Conv}	
Before showing convergence of the involved functions, we will state three lemmata needed in the sequel. The first two concern special matrix-valued cut-off functions as introduced in \cite{Bravin2024, OschmannPokorny2023} for $d=3$, and in \cite{NecasovaOschmann2023} for $d=2$.
\begin{lemma}\label{lemPhi}
    Let $d=3$ and $\delta>0$ such that $2 \eps^\alpha < \eps^{1+\delta}$. There exist matrix-valued functions $\Phi_\eps$ fulfilling
    \begin{align*}
        \Phi_\eps &\in W^{1,q}(D)\cap L^\infty(D) \text{ for any } q \in [1,\infty),\\
        \div \Phi_\eps &= 0,\\
        \Phi_\eps &=0 \text{ on } D\setminus D_\eps,\\
        \Phi_\eps &=\mathbb{I} \text{ on } D\setminus \bigcup_{i=1}^{N(\eps)} B_{\eps^{1+\delta}}(x_i(\eps)).
    \end{align*}
    Moreover, $\|\Phi_\eps\|_{L^\infty(D)}\lesssim 1$, and for any $1\leq q< \infty$,
    \begin{align*}
        \|\Phi_\eps-\mathbb{I}\|_{L^q(D)}^q &\lesssim \eps^{3\delta}, &&\|\nabla\Phi_\eps\|_{L^q(D)}^q \lesssim \eps^{(3-q)\alpha - 3} \begin{cases}
            |\eps^{(3-2q)(1+\delta-\alpha)} - 1| & \text{if } q\neq \frac32,\\
            |\log(\eps^{1+\delta-\alpha})| & \text{if } q=\frac32.
        \end{cases}
    \end{align*}
    In turn, for any $\psi\in C_c^\infty(D;\R^3)$ and any $q \geq 2$,
    \begin{align*}
        \|\nabla (\Phi_\eps \psi)-\Phi_\eps \nabla\psi\|_{L^q(D)}^q &\lesssim \eps^{(3-q)\alpha - 3} \|\psi\|_{L^\infty(D)}^q.
    \end{align*}
\end{lemma}

\begin{lemma}\label{lemPhi2}
    Let $d=2$, $\alpha > 2$ and $a_\eps = \exp(-\eps^{-\alpha})$. Then, there are matrix-valued functions $\Phi_\eps$ with
    \begin{align*}
        \Phi_\eps &\in W^{1,q}(D)\cap L^\infty(D) \text{ for any } q \in [1,\infty),\\
        \div \Phi_\eps &= 0,\\
        \Phi_\eps &=0 \text{ on } D\setminus D_\eps,\\
        \Phi_\eps &=\mathbb{I} \text{ on } D\setminus \bigcup_{i=1}^{N(\eps)} B_{\eps^\frac{\alpha}{2}}(x_i(\eps)).
    \end{align*}
    Moreover, $\|\Phi_\eps\|_{L^\infty(D)}\lesssim 1$, and for any $1\leq q< \infty$,
    \begin{align*}
        \|\Phi_\eps-\mathbb{I}\|_{L^q(D)}^q + \|\nabla \Phi_\eps\|_{L^2(D)}^2 \lesssim \eps^{\alpha-2}.
    \end{align*}
    In turn, for any $\psi\in C_c^\infty(D;\R^3)$ and any $q \leq 2$,
    \begin{align*}
        \|\nabla (\Phi_\eps \psi)-\Phi_\eps \nabla\psi\|_{L^q(D)} &\lesssim \eps^\frac{\alpha-2}{2} \|\psi\|_{L^\frac{2q}{2-q}(D)},
    \end{align*}
    with the convention that $1/0 = \infty$.
\end{lemma}

The third lemma is an easy convergence result we will use several times in the sequel:
\begin{lemma}\label{lem:equiConv}
    Let $\{f_\eps\}_{\eps \in (0,1)} \subset L^1((0,T) \times D)$ be a real-, vector-, or matrix-valued sequence of functions that is equi-integrable, and let $\Phi_\eps$ be the functions from Lemma~\ref{lemPhi} or \ref{lemPhi2}. Then
    \begin{align*}
        (\Phi_\eps - \mathbb I) f_\eps \to 0 \text{ in } L^1((0,T) \times D).
    \end{align*}
\end{lemma}
\begin{proof}
    By the definition of $\Phi_\eps$, we see that
    \begin{align*}
        \Phi_\eps - \mathbb I = 0 \text{ on } D \setminus \bigcup_{i =1}^{N(\eps)} B_{\eps^\sigma (x_i(\eps))},
    \end{align*}
    where $\sigma=1+\delta$ if $d=3$, and $\sigma=\alpha/2$ if $d=2$. Moreover,
    \begin{align*}
        \Big|(0,T) \times \bigcup_{i=1}^{N(\eps)} B_{\eps^\sigma (x_i(\eps))} \Big| \lesssim \eps^{d(\sigma-1)} \to 0
    \end{align*}
    since always $\sigma >1$. Now, let $\text{\ss} >0$ be arbitrary. Since $f_\eps$ is equi-integrable, there exists $\text{\l} >0$ such that for any set $M \subset (0,T) \times D$ with $|M| <\text{\l}$, we have
    \begin{align*}
        \int_M |f_\eps| \dd x \dd t <\text{\ss}.
    \end{align*}
    Together with the fact that $\Phi_\eps \in L^\infty(D)$ is uniformly bounded, we find for $\eps$ small enough that
    \begin{align*}
        \int_0^T \int_D | (\Phi_\eps - \mathbb I) f_\eps | \dd x \dd t &= \int_0^T \int_{\bigcup_{i=1}^{N(\eps)} B_{\eps^\sigma (x_i(\eps))}} | (\Phi_\eps - \mathbb I) f_\eps | \dd x \dd t \\
        &\lesssim \int_0^T \int_{\bigcup_{i=1}^{N(\eps)} B_{\eps^\sigma (x_i(\eps))}} | f_\eps | \dd x \dd t <\text{\ss},
    \end{align*}
    thus finishing the proof since $\text{\ss}>0$ was arbitrary.
\end{proof}


\subsection{Limiting functions}
As a consequence of the uniform bounds given in Section~\ref{sec:Bds}, we obtain, at least for a subsequence,
\begin{align*}
    \tvu_\eps &\weak \vu \ \text{weakly in} \ L^2(0,T;W_0^{1,2}(D)),\\
    \tvr_\eps &\weak^\ast \rho \ \text{weakly-($\ast$) in} \ L^\infty(0,T; L^\infty(D)),\\
    p(\tvr_\eps) &\weak \overline{p(\rho)} \ \text{weakly in} \ L^1((0,T)\times D),
\end{align*}

where the convergence of the pressure follows from Dunford-Pettis theorem.


\subsection{Limit in the continuity equation}\label{convCE}
Repeating the arguments given in \cite{LuSchwarzacher2018, PokornySkrisovsky2021}, we recover
\begin{align}\label{approx_cont}
    &\del_t \tvr_\eps + \div (\tvr_\eps\tvu_\eps) = 0 \ \text{in} \ \Test'((0,T) \times \R^3)
\end{align}
as well as, after the limit passage,
\begin{align*}
    \del_t \tvr + \div (\tvr \tvu) = 0 \ \text{in} \ \Test'((0,T) \times \R^3).
\end{align*}
Indeed, from the uniform bounds on $\tvr_\eps$ and $\tvu_\eps$ derived in Section~\ref{sec:Bds}, we have
\begin{align*}
    \|\tvr_\eps\tvu_\eps\|_{L^\infty(0,T;L^2(D))}\leq \|\sqrt{\tvr_\eps}\|_{L^\infty(0,T;L^2(D))} \|\sqrt{\tvr_\eps}\tvu_\eps\|_{L^\infty(0,T;L^2(D))}\leq C.
\end{align*}
Moreover, by \eqref{approx_cont}, we have
\begin{align*}
    \del_t \tvr_\eps \text{ bounded in } L^2(0,T;W^{-1,2}(D)).
\end{align*}
Applying \cite[Lemma~5.1]{Lions1998} now shows
\begin{align*}
    \tvr_\eps\tvu_\eps \to \rho\vc u \text{ in } \mathcal{D}'((0,T)\times D).
\end{align*}
Furthermore, an Aubin-Lions-Simon type argument yields for any $q \in [1,\infty)$
\begin{align*}
    \tvr_\eps \to \rho \text{ in } C_w(0,T;L^q(D)),\quad \tvr_\eps \tvu_\eps \to \rho \vc u \text{ in } C_w(0,T;L^2(D)).
\end{align*}
A similar argument applies to $\tvr_\eps\tvu_\eps\otimes\tvu_\eps$, which is needed in the limit passage in the momentum equation below.
As a consequence of \cite[Lemma~6.9]{Novotny2004}, the couples $(\tvr_\eps, \tvu_\eps)$ as well as $(\tvr, \tvu)$ also fulfil the renormalized continuity equation in $\Test'((0,T) \times \R^3)$, that is, the equations given in \eqref{renCE}.


\subsection{Limit in the momentum equation}
By the arguments from Section~\ref{convCE} and the strong convergence of $\Phi_\eps$ from Lemma~\ref{lemPhi} and \ref{lemPhi2}, we find\footnote{The fact that $\tvr_\eps \tvu_\eps$ needs $\Phi_\eps^\top$, but $\tvr_\eps \tvu_\eps \otimes \tvu_\eps$ sees $\Phi_\eps$, follows from the expansion of gradients, see \eqref{expanGrad} below.}
\begin{align*}
    \begin{split}
        \Phi_\eps^\top \tvr_\eps \tvu_\eps &\to \rho \vc u \text{ in } C_w(0,T;L^r(D)), \quad r< 2,\\
        \Phi_\eps \tvr_\eps \tvu_\eps \otimes \tvu_\eps &\to \rho \vc u \otimes \vc u \text{ in } \mathcal{D}'((0,T)\times D).
    \end{split}
\end{align*}

For $\phi \in C_c^\infty([0,T)\times D)$, we use $\Phi_\eps \phi \in [W_0^{1,q} \cap L^\infty]([0,T) \times D_\eps)$ as a proper test function in the momentum equation \eqref{NSE}$_2$. Since $\Phi_\eps = 0$ on the holes, we can prolong all functions by zero to the whole of $D$ and obtain
\begin{align*}
    0&=\int_D \tilde{\vc m}_{\eps, 0}\cdot  \Phi_\eps \phi(0,\cdot) \dd x + \int_0^T\int_D \tvr_\eps\tvu_\eps \cdot \Phi_\eps \del_t\phi \dd x \dd t + \int_0^T\int_D \tvr_\eps \tvu_\eps \otimes \tvu_\eps : \nabla (\Phi_\eps \phi) \dd x \dd t\\
    &\quad + \int_0^T\int_D p(\tvr_\eps) \div(\Phi_\eps \phi) \dd x \dd t - \int_0^T \int_D \mathbb{S}(\nabla\tvu_\eps):\nabla(\Phi_\eps \phi) \dd x \dd t + \int_0^T\int_D \tvr_\eps \vc f\cdot\Phi_\eps \phi \dd x \dd t\\
    &=: \sum_{j=1}^6 I_j.
\end{align*}

Since $\tvr_\eps \weak^\ast \rho$ weakly-($\ast$) in $L^\infty(0,T; L^\infty(D))$, in particular in $L^\infty(0,T;L^p(D))$ for any $p \in [1,\infty]$, we follow the same arguments as in \cite{OschmannPokorny2023, NecasovaOschmann2023} to easily pass to the limits in each $I_j$ except $I_4$, giving the same restrictions on the holes' size $\alpha > d$ when, formally, setting $\gamma=\infty$ in the above references. We therefore focus just on the pressure term $I_4$, which is different in our case. However, since $\div \Phi_\eps = 0$, the pressure integral reads
\begin{align*}
    \int_0^T \int_D p(\tvr_\eps) \div (\Phi_\eps \phi) \dd x \dd t &= \int_0^T \int_D p(\tvr_\eps) \Phi_\eps:\nabla \phi \dd x \dd t \\
    &= \int_0^T \int_D p(\tvr_\eps) \mathbb{I} : \nabla \phi \dd x \dd t + \int_0^T \int_D p(\tvr_\eps) (\Phi_\eps - \mathbb{I}) : \nabla \phi \dd x \dd t.
\end{align*}
Since $p(\tvr_\eps) \weak \overline{p(\rho)}$ weakly in $L^1((0,T) \times D)$, we have
\begin{align*}
    \int_0^T \int_D p(\tvr_\eps) \mathbb{I} : \nabla \phi \dd x \dd t \to \int_0^T \int_D \overline{p(\rho)} \mathbb{I} : \nabla \phi \dd x \dd t = \int_0^T \int_D \overline{p(\rho)} \div \phi \dd x \dd t.
\end{align*}
The remaining part converges to zero thanks to Lemma~\ref{imprPress} and \ref{lem:equiConv}, yielding in total
\begin{align*}
    \int_0^T \int_D p(\tvr_\eps) \Phi_\eps : \nabla \phi \dd x \dd t \to \int_0^T \int_D \overline{p(\rho)} \div \phi \dd x \dd t.
\end{align*}
Therefore, for $\phi\in C_c^\infty([0,T)\times D; \R^d)$ we obtain 
\begin{align}\label{wkMom-lim}
    \begin{split}
        &\int_0^T\int_{D} \rho \vu \cdot \del_t \phi \dd x \dd t + \int_0^T\int_{D} \rho \vu\otimes \vu : \nabla \phi \dd x \dd t + \int_0^T\int_{D} \overline{p(\rho)} \div \phi \dd x \dd t \\
        &- \int_0^T\int_{D} \ms(\nabla \vu):\nabla \phi \dd x \dd t + \int_0^T\int_{D} \rho \vc f \cdot \phi \dd x \dd t = -\int_{D} \vc m_0 \cdot \phi(0,\cdot) \dd x .
    \end{split}
\end{align}
Once we showed that $\overline{p(\rho)} = p(\rho)$, this finishes the proof of Theorem~\ref{thm1}. The proof of this last identity is carried out in the next section.


\section{Strong convergence of the density}\label{sec:strDens}

Let us first expand the expressions $\div(\alpha \mathbb{A} {\vc g})$ and $\nabla(\alpha \mathbb{A} {\vc g})$, where
\begin{align*}
\alpha(x) \in \mathbb{R}, && {\vc g}(x):=\lr{g_i}_{i=1}^{d}  \in \mathbb{R}^{d}, && \mathbb{A}(x):=\lr{A_{ij}}_{i,j=1}^{d} \in \mathbb{R}^{d\times d} .
\end{align*}
Then, using Einstein summation convention and our definition of matrix-divergence and vector-gradient, the term $\div(\alpha \mathbb{A} {\vc g})$ will be expanded as
\begin{align}
	\div(\alpha \mathbb{A} {\vc g})
	&= \del_i (\alpha A_{ij} g_j) \notag \\
	&= \del_i \alpha A_{ij} g_j + \alpha \del_i A_{ij} g_j + \alpha A_{ij} \del_i g_j \notag \\
    &= (\nabla \alpha \otimes \vc g) : \mathbb{A} + \alpha (\div \mathbb{A}) \cdot \vc g + \alpha \mathbb{A} : \nabla \vc g . \label{expanDiv}
\end{align}
Similarly, the full gradient expands as
\begin{align}
	\nabla(\alpha \mathbb{A} {\vc g})
	&= \lr{ \del_i (\alpha A_{jk} g_k) }_{ij} \notag \\
    &= \lr{ \del_i \alpha A_{jk} g_k + \alpha \del_i A_{jk} g_k + \alpha A_{jk} \del_i g_k }_{ij} \notag \\
    &= (\nabla \alpha \otimes \vc g) \mathbb{A}^\top + \alpha \vc g \del \mathbb{A}^\top + \alpha (\nabla \vc g) \mathbb{A}^\top , \label{expanGrad}
\end{align}

\noindent
where we define
\begin{align}\label{defVecDer}
\lr{ \vc g \del \mathbb{A}^\top }_{ij} := g_k \del_i A_{jk} 
\end{align}

The general idea to prove the strong convergence of the density is nowadays standard and has been explained in several articles (see, e.g., \cite{Novotny2004, LuSchwarzacher2018}). It relies on the compactness of the so-called effective viscous flux. More generally, the Lions identity usually provides a way to directly apply Lions’ result on the weak continuity of the effective viscous pressure to obtain, in weak form,
\begin{align}\label{evf-formal}
	\overline{\rho p(\rho)}- \lr{\frac{2(d-1)}{d} \mu+ \mu_1 } \overline{\rho \div \vu} = \rho	\, \overline{ p(\rho)}- \lr{\frac{2(d-1)}{d} \mu+ \mu_1 } {\rho \div \vu}.  
\end{align}  
Here, $\overline{\rho p(\rho)}$ denotes the weak limit of $\tvr_\eps p(\tvr_\eps)$ in $L^1 L^1$. In our case, the equi-integrability of the term $\tvr_\eps p(\tvr_\eps)$ follows from the observation already made in \eqref{eqiint-pr} and the fact that $\tvr_\eps \leq \ovr$; indeed, since also $s |\log s| \leq 1 + s^2$, we have
\begin{align}\label{equiintRhoP}
	|\tvr_\eps p(\tvr_\eps) \log \lr{\tvr_\eps p(\tvr_\eps)}| \lesssim \tvr_\eps p(\tvr_\eps) |\log \tvr_\eps| + \tvr_\eps p(\tvr_\eps) |\log p(\tvr_\eps) | \lesssim p(\tvr_\eps) + p(\tvr_\eps) |\log p(\tvr_\eps)| ,
\end{align}
and the latter is uniformly bounded in $L^1 L^1$ thanks to Lemma~\ref{imprPress}.\\

From the renormalized continuity equations fulfilled by $(\rhoe, \vu_\eps)$ and $(\rho, \vu)$, we have 
\begin{align}\label{sc-dens1}
    \partial_t (\tvr_\eps \log \tvr_\eps )+ \div \lr{\tvr_\eps \log (\tvr_\eps) \tvu_\eps}+ \rhoe \div \tvu_\eps =0 \text{ in } \mathcal{D}^{\prime}((0,T)\times D)
\end{align}
and 
\begin{align}\label{sc-dens2}
     \partial_t ( \rho \log \rho )+ \div \lr{ \rho \log (\rho) \vu}+ \rho \div \vu =0 \text{ in } \mathcal{D}^{\prime}((0,T)\times D).
\end{align}
Letting $\eps \rightarrow 0$ in \eqref{sc-dens1} and subtracting \eqref{sc-dens2}, it holds that
\begin{align}\label{sc-dens3}
    \partial_t (\overline{\rho \log \rho}- \rho \log \rho )+ \div \lr{ \overline{\rho \log (\rho) \vu}- \rho \log (\rho) \vu}+ \lr{\overline{\rho \div \vu}-\rho \div \vu} =0 \text{ in } \mathcal{D}^{\prime}((0,T)\times D).
\end{align}
Then, by employing a sequence of test functions $\{\varphi_n(t)\}_{n \in \N} \subset C_c^\infty(0,T)$ that converges strongly to $1$ in $L^2(0,T)$ in the weak formulation of \eqref{sc-dens3}, we obtain, for a.e.~$\tau \in (0,T]$,
\begin{align*}
	\int_D \lr{\overline{\rho \log \rho}- \rho \log \rho}(\tau) \dd x + \int_0^\tau \int_D \lr{ \overline{\rho \div \vu} - \rho \div \vu } \dd x \dd t = 0 .
\end{align*}

We conclude the proof by establishing the identity \eqref{evf-formal}, which together with the monotonicity of $p$ gives  
\begin{align*}
	\lr{\frac{2(d-1)}{d} \mu+ \mu_1 } \int_0^\tau \int_D  \lr{ \overline{\rho \div \vu} - \rho \div \vu}  \dd x\dd t = \int_0^\tau \int_D \lr{ \overline{\rho p( \rho)}- \rho \, \overline{p( \rho)} } \dd x \dd t \geq 0. 
\end{align*}  

Finally, the strict convexity of the function $s \mapsto s \log s$ on $[0,\infty)$ ensures the strong convergence of the density.

To make things rigorous, we will prove the following lemma: 
\begin{lemma}\label{evf-lem}
    There holds the identity 
    \begin{align}\label{evf-id-weak}
        \begin{split}
        \int_0^T  \int_D \psi(t)\zeta(x) \lr{ \overline{\rho p(\rho)} - \lr{\frac{2(d-1)}{d} \mu+ \mu_1 } \overline{\rho \div \vu} }  \dd x  \dd t
        \\ =
        \int_0^T \int_D \psi(t) \zeta(x) \lr{ \rho \, \overline{p(\rho)} - \lr{\frac{2(d-1)}{d} \mu+ \mu_1 } \rho  \div \vu }  \dd x  \dd t,
        \end{split}
    \end{align}
    where 
    $\psi \in C^\infty_c (0,T)$ and $\zeta \in C_c^\infty(D)$.
\end{lemma}
To prove Lemma~\ref{evf-lem}, the key idea is to use in the momentum equation \eqref{wkMom} the test function  
\begin{align}\label{phieps}
    \phi_\eps(t,x)= \psi(t)  \zeta(x)  \Phi_\eps  \nabla \Delta^{-1}[\mathbf{1}_{D} \tilde \varrho_\eps],
\end{align}  
and in the limit equation \eqref{wkMom-lim} the corresponding test function  
\begin{align*}
    \phi(t,x)= \psi(t)  \zeta(x)  \nabla \Delta^{-1}[\mathbf{1}_{D} \varrho],
\end{align*}  
where both $\mathbf u_\eps$ and $\mathbf u$ are extended by zero outside $D$ to the whole of $\mathbb{R}^d$, $\psi$ and $\zeta$ belong to the class mentioned in Lemma~\ref{evf-lem}, $\Phi_\eps$ is as in Lemma~\ref{lemPhi} for $d=3$ (or, equivalently, Lemma~\ref{lemPhi2} for $d=2$), and $\Delta^{-1}$ is the Fourier multiplier with symbol $-|\xi|^2$.


\subsection{Justification of use of test functions}
Let us explain why we are allowed to use \eqref{phieps} as test function in \eqref{wkMom}. First, from the properties of the Riesz operator $\nabla^2 \Delta^{-1}$, we find for any $q \in [1,\infty)$
\begin{align*}
    \|\nabla \Delta^{-1} [\mathbf{1}_D \tvr_\eps]\|_{L^\infty W^{1,q}} \lesssim \|\tvr_\eps\|_{L^\infty L^q} \lesssim 1,
\end{align*}
in particular, $\nabla \Delta^{-1}[\mathbf{1}_D \tvr_\eps] \in L^\infty L^\infty$ uniformly in $\eps$. This together with the uniform bound on $\Phi_\eps$ in $L^\infty$ leads to
\begin{align}\label{bdsphi}
    \|(\nabla \Delta^{-1}[\mathbf{1}_D \tvr_\eps]) \del \Phi_\eps\|_{L^\infty L^2} + \|\Phi_\eps \nabla^2 \Delta^{-1}[\mathbf{1}_D \tvr_\eps]\|_{L^\infty L^q} \lesssim 1
\end{align}
for any finite $q$. \par 

From the momentum equation \eqref{wkMom}, the continuity equation \eqref{renCE}, and the fact that $\div \Phi_\eps = 0$, together with \eqref{expanDiv}--\eqref{defVecDer}, we take $\phi_\eps(t,x) = \psi(t) \zeta(x) \Phi_\eps \nabla \Delta^{-1} [\mathbf{1}_D \tvr_\eps]$ as test function to obtain
    \begin{align}\label{conv1bd}
    \sum_{i=1}^{11} \mathcal{K}_i :=& \int_0^T \partial_t \psi(t) \lr{\int_D \tvr_\eps\tvu_\eps \cdot \zeta(x) \Phi_\eps \nabla \Delta^{-1}[\mathbf{1}_{D} \tvr_\eps]\dd x }\dd t \nonumber \\
    &-  \int_0^T  \psi(t) \lr{\int_D \tvr_\eps\tvu_\eps \cdot \zeta(x) \Phi_\eps  \nabla \Delta^{-1} \div (\tvr_\eps \tvu_\eps) \dd x }\dd t \nonumber\\ 
    &+ \int_0^T \psi(t) \lr{ \int_D \tvr_\eps \tvu_\eps \otimes \tvu_\eps : \lr{ \nabla \zeta(x) \otimes \nabla \Delta^{-1}[\mathbf{1}_{D} \tvr_\eps]  }\Phi_\eps^\top  \dd x} \dd t\nonumber\\
    &+ \int_0^T \psi(t) \lr{\int_D \tvr_\eps \tvu_\eps \otimes \tvu_\eps :  \zeta(x) \lr{ \nabla \Delta^{-1}[\mathbf{1}_{D} \tvr_\eps] } \del \Phi_\eps^\top \dd x }\dd t\nonumber\\ 
    &+ \int_0^T \psi(t) \lr{\int_D \tvr_\eps \tvu_\eps \otimes \tvu_\eps : \zeta(x) (\nabla^2 \Delta^{-1}[\mathbf{1}_{D} \tvr_\eps]) \Phi_\eps^\top \dd x} \dd t \nonumber\\
    &+\int_0^T \psi(t) \lr{ \int_D p(\tvr_\eps) \Phi_\eps : \nabla \zeta(x) \otimes \nabla \Delta^{-1}[\mathbf{1}_{D} \tvr_\eps] \dd x }\dd t \nonumber\\
    &+\int_0^T \psi(t) \lr{ \int_D p(\tvr_\eps) \zeta(x) \Phi_\eps :  \nabla^2 \Delta^{-1}[\mathbf{1}_{D} \tvr_\eps] \dd x }\dd t \nonumber\\
    &-\int_0^T \psi(t) \lr{ \int_D \mathbb{S}(\nabla\tvu_\eps) : \lr{ \nabla \zeta(x) \otimes \nabla \Delta^{-1}[\mathbf{1}_{D} \tvr_\eps]  }\Phi_\eps^\top  \dd x} \dd t \nonumber\\
    &- \int_0^T \psi(t) \lr{\int_D \mathbb{S}(\nabla\tvu_\eps):  \zeta(x) \lr{ \nabla \Delta^{-1}[\mathbf{1}_{D} \tvr_\eps] } \del \Phi_\eps^\top \dd x }\dd t \nonumber\\ 
    &- \int_0^T \psi(t) \lr{\int_D \mathbb{S}(\nabla\tvu_\eps) : \zeta(x) \lr{\nabla^2 \Delta^{-1}[\mathbf{1}_{D} \tvr_\eps] } \Phi_\eps^\top \dd x} \dd t \nonumber \\
    &- \int_0^T\psi(t) \lr{\int_D \tvr_\eps \vc f\cdot  \zeta(x) \Phi_\eps \nabla \Delta^{-1}[\mathbf{1}_{D} \tvr_\eps] \dd x }\dd t =0.
\end{align}
First, we need to check that all the terms in \eqref{conv1bd} are well defined. The core difficulty is to check whether the pressure term 
\begin{align}\label{jednak}
    \int_0^T \psi(t) \lr{ \int_D p(\tvr_\eps) \zeta(x) \Phi_\eps :  \nabla^2 \Delta^{-1}[\mathbf{1}_{D} \tvr_\eps] \dd x }\dd t
\end{align}
is well defined since $p(\tvr_\eps)$ is just uniformly bounded in $L^1(0,T;L (\log L)^{1+\kappa}(D))$. \par 

The functional framework developed in Appendix~\ref{sec:AppB} is essential to obtain a bound for the above-mentioned term. At first we observe that 
\begin{align*}
    \int_0^T \psi(t) \int_D p(\tvr_\eps) \zeta(x) \Phi_\eps : \nabla^2 \Delta^{-1}[\mathbf{1}_{D} \tvr_\eps] \dd x \dd t 
    &= \int_0^T \psi(t) \int_{\mathbb{R}^n} \mathbf{1}_{D} p(\tvr_\eps) \zeta(x) \Phi_\eps : \nabla^2 \Delta^{-1}[\mathbf{1}_{D} \tvr_\eps] \dd x \dd t .
\end{align*}
Applying the relations \eqref{loglexpl}--\eqref{time_BMO_exp}, we justify the well-definedness of this integral as follows:

\begin{itemize}[leftmargin=*]
    \item Since the density sequence $\{\tvr_\eps\}_{\eps \in (0,1)}$ is bounded in $L^\infty(0,T; L^\infty(D))$, the singular integral operator $\nabla^2 \Delta^{-1}$ applied to $\mathbf{1}_{D}\tvr_\eps$ satisfies
\begin{align*}
    \nabla^2 \Delta^{-1}[\mathbf{1}_{D}\tvr_\eps] \in L^\infty_{w-(\ast)}(0,T; \BMO(\mathbb{R}^d)).
\end{align*}
Recalling the embedding for the exponential Orlicz scale,
\begin{equation*}
    L^\infty_{w-(\ast)}(0,T; \BMO(\mathbb{R}^d)) \hookrightarrow L^\infty_{w-(\ast)}(0,T; \exp L^{\frac{1}{1+\kappa}}(\mathbb{R}^d)),
\end{equation*}
it follows that $\nabla^2 \Delta^{-1}[\mathbf{1}_{D}\tvr_\eps]$ is uniformly bounded in $L^\infty_{w-(\ast)}(0,T; \exp L^{\frac{1}{1+\kappa}}(\mathbb{R}^d))$.
    \item Given $p(\tvr_\eps) \in L(\log L)^{1+\kappa}((0,T) \times D) \hookrightarrow L^1(0,T; L(\log L)^{1+\kappa}(D))$, the spatial integrand is realized via the duality pairing:
    \begin{align*}
      \int_{\R^d} p(\tvr_\eps) \Phi_{\eps}\zeta(x) : \nabla^2 \Delta^{-1} [\mathbf{1}_D \tvr_\eps] \dd x = \left\langle p(\tvr_\eps)\zeta(x), \Phi_{\eps} : \nabla^2 \Delta^{-1} [\mathbf{1}_D \tvr_\eps] \right\rangle_{L(\log L)^{1+\kappa}(\mathbb{R}^d) , \exp L^{\frac{1}{1+\kappa}}(\mathbb{R}^d)}.
    \end{align*}
    Note that while the pressure $p(\tvr_\eps)$ is supported on $\overline{D}$, the singular integral operator $\nabla^2 \Delta^{-1}$ is non-local; hence, the duality is rigorously defined over $\mathbb{R}^d$ using the extension $\mathbf{1}_D \tvr_\eps(t) \in L^\infty(\mathbb{R}^d)$.

    \item The \textit{separability} of the Orlicz space $L(\log L)^{1+\kappa}(\mathbb{R}^d)$ ensures that the duality mapping 
    \begin{equation*}
        t \mapsto \psi(t)\left\langle p(\tvr_\eps(t, \cdot))\zeta(\cdot), \Phi_\eps : \nabla^2 \Delta^{-1} [\mathbf{1}_D \tvr_\eps(t, \cdot)] \right\rangle_{L(\log L)^{1+\kappa}(\mathbb{R}^d) , \exp L^{\frac{1}{1+\kappa}}(\mathbb{R}^d)}
    \end{equation*}
    is a measurable function of time. \par  Since $p(\tvr_\eps) \in L^1(0,T; L(\log L)^{1+\kappa})$ and $\nabla^2 \Delta^{-1}[\mathbf{1}_D \tvr_\eps] \in L^\infty(0,T; \exp L^{\frac{1}{1+\kappa}})$, the integral is a well-defined element of $L^1(0,T)$. Thus, the term \eqref{jednak} is realized as
    \begin{align*}
        &\int_0^T \psi(t) \int_{\mathbb{R}^d}  p(\tvr_\eps) \zeta(x) \Phi_\eps : \nabla^2 \Delta^{-1}[\mathbf{1}_{D} \tvr_\eps] \dd x \dd t \\
        &\quad = \left\langle \psi p(\tvr_\eps) \zeta, \Phi_\eps : \nabla^2 \Delta^{-1} [\mathbf{1}_D \tvr_\eps] \right\rangle_{L^1(0,T; L(\log L)^{1+\kappa}) , L^\infty_{w-(\ast)}(0,T; \exp L^{\frac{1}{1+\kappa}})}.
    \end{align*}
\end{itemize}
\textbf{Remark.} Note that $\kappa=0$ is sufficient to ensure the well-definedness of the pressure term. In this baseline case, the pressure $p(\rhoe) \in L^1(0,T; L \log L(D))$ pairs directly with the Riesz operator in $L^\infty_{w-(\ast)}(0,T; \BMO(D))$. Since the latter embeds continuously into $L^\infty_{w-(\ast)}(0,T; \exp L(D))$, the duality pairing is rigorously justified. In the next section, we rigorously justify the limit passage in this term, where $\kappa>0$ will play a crucial role.


\subsection{Limit passage}
We follow \cite[Section~3.4.2]{LuSchwarzacher2018} to derive the convergence of $ \phi_\eps \lr{= \psi(t)  \zeta(x)  \Phi_\eps  \nabla \Delta^{-1}[\mathbf{1}_{D} \tilde \varrho_\eps]}$ by suitably adapting the arguments for our regularity of $\lr{\tvr_\eps,\tvu_\eps}$. At first, using the continuity equation 
\begin{equation*}
    \partial_t \tvr_\eps + \div(\tvr_\eps \vu_\eps) = 0 \quad \text{in } \mathcal{D}^\prime((0,T)\times \mathbb{R}^d),
\end{equation*}
along with the uniform velocity in $L^2(0,T;W^{1,2}_0(D))$ and density in $\tvr_\eps \in L^\infty(0,T; L^\infty(D))$, we have  $\{\tvr_\eps \vu_\eps\}_{\eps \in (0,1)}$ is uniformly bounded in $L^2(0,T; L^6(D))$. Thus, the sequence of time derivatives $\{\partial_t \tvr_\eps\}_{\eps \in (0,1)}$ is uniformly bounded in $L^2(0,T; W^{-1,6}(D))$. By the fact that $\nabla \Delta^{-1}$ maps $W^{-1,6}(D)$ into $L^6(D)$,  we deduce
\begin{align}\label{lpevf1}
    \left\| \partial_t \nabla \Delta^{-1} [\mathbf{1}_D \tilde \varrho_\eps] \right\|_{L^2(0,T; L^6(D))} &\le C \|\tilde \varrho_\eps\|_{L^\infty(0,T; L^\infty(D))} \|\tilde \vu_\eps\|_{L^2(0,T; W_0^{1,2}(D))} \le C.
\end{align}

By virtue of the uniform spatial bounds on the density and the temporal control provided by \eqref{lpevf1}, we can invoke the Aubin--Lions--Simon Compactness Theorem. This establishes that the family $\left\{ \nabla \Delta^{-1} [\mathbf{1}_D \tilde \varrho_\eps] \right\}_{\eps \in (0,1)}$ is strongly precompact in $C_w(0, T; L^r(D;\mathbb{R}^d))$ for any $1<r<\infty$.

We now establish the strong convergence of the full test function sequence 
\begin{equation*}
    \phi_\eps = \psi(t) \zeta(x) \Phi_\eps \nabla \Delta^{-1}[\mathbf{1}_{D} \tilde \varrho_\eps]
\end{equation*}
by analyzing the structural profiles of the cutting matrices in both dimensions. To this end, we decompose the convergence error as
\begin{align*}
    \phi_\eps - \psi \zeta \nabla \Delta^{-1} [\mathbf{1}_D \rho] = \psi \zeta \Phi_\eps \left( \nabla \Delta^{-1} [\mathbf{1}_D \tilde \varrho_\eps] - \nabla \Delta^{-1} [\mathbf{1}_D \rho] \right) + \psi \zeta (\Phi_\eps - \mathbb{I}) \nabla \Delta^{-1} [\mathbf{1}_D \rho].
\end{align*}

Thanks to the strong convergences of $\nabla \Delta^{-1} [\mathbf{1}_D \tvr_\eps]$ and $\Phi_\eps$ (see Lemma~\ref{lemPhi} and \ref{lemPhi2}), we readily see that
\begin{equation*}
    \phi_\eps \to \psi(t) \zeta(x) \nabla \Delta^{-1} [\mathbf{1}_D \rho] \quad \text{strongly in } C_w(0,T; L^r(D)),
\end{equation*}
for any $1 < r < \infty$.\\

All the terms $\mathcal{K}_i$ in \eqref{conv1bd} except of $\mathcal{K}_6$ and $\mathcal{K}_7$ can be handled in a classical way once we split $\Phi_\eps = (\Phi_\eps - \mathbb{I}) + \mathbb{I}$. The identity part is handled as in \cite{LuSchwarzacher2018}, whereas the difference part converges to zero thanks to the estimates on $(\Phi_\eps - \mathbb{I})$ from Lemma~\ref{lemPhi} and \ref{lemPhi2}, and the bounds obtained in \eqref{bdsphi}.\\
    
    The most technical point for us is the terms $\mathcal{K}_6, \mathcal{K}_7$ containing the pressure, which differs in our case and thus needs special care. First, for $\mathcal{K}_{6}$, using equi-integrability and similar arguments as before, we rewrite
    \begin{align*}
        &\int_0^T \psi(t) \lr{ \int_D p(\tvr_\eps) \Phi_\eps : \nabla \zeta(x) \otimes \nabla \Delta^{-1}[\mathbf{1}_{D} \tvr_\eps] \dd x }\dd t \\
        &= \int_0^T \psi(t) \lr{ \int_D p(\tvr_\eps) \mathbb{I} : \nabla \zeta(x) \otimes \nabla \Delta^{-1}[\mathbf{1}_{D} \tvr_\eps] \dd x }\dd t \\
        &\qquad + \int_0^T \psi(t) \lr{ \int_D p(\tvr_\eps) (\Phi_\eps - \mathbb{I}) : \nabla \zeta(x) \otimes \nabla \Delta^{-1}[\mathbf{1}_{D} \tvr_\eps] \dd x }\dd t 
    \end{align*}
and see that the latter integral converges to zero thanks to Lemma~\ref{lem:equiConv}. Similarly for $\mathcal{K}_{7}$, we rewrite
    \begin{align*}
        &\int_0^T \psi(t) \lr{ \int_D p(\tvr_\eps) \zeta(x) \Phi_\eps :  \nabla^2 \Delta^{-1}[\mathbf{1}_{D} \tvr_\eps] \dd x }\dd t \\
        &= \int_0^T \psi(t) \lr{ \int_D \zeta(x) \tvr_\eps p(\tvr_\eps) \dd x }\dd t \\
        &\qquad + \int_0^T \psi(t) \lr{ \int_D p(\tvr_\eps) \zeta(x) (\Phi_\eps - \mathbb{I}) :  \nabla^2 \Delta^{-1}[\mathbf{1}_{D} \tvr_\eps] \dd x }\dd t =: \mathcal{T}_1 + \mathcal{T}_2.
    \end{align*}
    For $\mathcal{T}_1$, we use the equi-integrability of the sequence $\{\tvr_\eps p(\tvr_\eps)\}_{\eps \in (0,1)}$ from \eqref{equiintRhoP} and Dunford-Pettis theorem. In contrast, the convergence of the term $\mathcal{T}_2$ is quite subtle due to the presence of the term $(\Phi_\eps - \mathbb{I})$; we provide a rigorous proof, relying on the technical results developed in Appendix~ \ref{sec:AppB} (more precisely Appendix~\ref{sec:AppA}).\\
    The aim is to prove that
\begin{align}\label{eq1}
   \lim_{\eps \to 0}\mathcal{T}_2= \lim_{\eps \to 0} \int_0^T \psi(t) \int_D p(\tvr_\eps) \zeta(x) (\Phi_\eps - \mathbb{I}) : \nabla^2 \Delta^{-1} [\mathbf{1}_D \tvr_\eps] \dd x \dd t  = 0.
\end{align}
We remark that this would directly follow from Lemma~\ref{lem:equiConv} if we would have $\nabla^2 \Delta^{-1} [\mathbf{1}_D \tvr_\eps] \in L^\infty(0,T; L^\infty(\R^d))$ uniformly, however, the operator $\nabla^2 \Delta^{-1}$ just maps $L^p$ to $L^p$ for any \emph{finite} $p>1$. We overcome this drawback by considering the spaces $L (\log L)^{1+\kappa}$ and $\BMO$. \\ 

\paragraph{\textbf{The Role of $\kappa>0$ in $L (\log L)^{1+\kappa}$}} To apply Lemma~\ref{lem:equiConv} and conclude that $(\Phi_\eps - \mathbb{I}) f_\eps \to 0$, the sequence $f_\eps$ must be \textit{equi-integrable} in $L^1((0,T) \times D)$. In our context, this requires the equi-integrability of the product
\begin{equation*}
    f_\eps := p(\tvr_\eps)  \nabla^2 \Delta^{-1} [\mathbf{1}_D \tvr_\eps].
\end{equation*}
This requirement poses a significant mathematical hurdle. If we only possess bounds in the critical spaces $L \log L$ and $\BMO$, the product resides merely in $L^1$, which does not inherently guarantee equi-integrability. Indeed, a simple counter-example (e.g., a sequence of functions concentrating into a Dirac mass, detailed in Appendix~\ref{lologl-bmo-ce}) shows that mere $L^1$ boundedness is insufficient for the limit passage in Lemma~\ref{lem:equiConv}. \par 
Consequently, the $L(\log L)^{1+\kappa}$ hierarchy is indispensable. By lifting the pressure integrability to $L(\log L)^{1+\kappa}$ for $\kappa > 0$, we ensure that the product $p(\tvr_\eps) \nabla^2 \Delta^{-1} [\mathbf{1}_D \tvr_\eps]$ gains a slight logarithmic ``margin" beyond $L^1$. This improved regularity provides the uniform integrability necessary to satisfy the hypotheses of Lemma~\ref{lem:equiConv}, thereby allowing us to conclude \eqref{eq1}.

Another important observation is that the duality $$L^1(0,T; L (\log L)^{1+\kappa})\times L^\infty_{w-(\ast)}(0,T; \BMO)$$ only guarantees that the integrand is in $L^1(0,T)$, which is insufficient for passing to the limit. We must utilize the full space-time $L (\log L)^{1+\kappa}((0,T) \times D)$ regularity of the pressure to ensure equi-integrability in time, thereby preventing mass concentration and justifying the use of Lemma~\ref{lem:equiConv}. In the following, we provide a rigorous justification for this.

We define the space-time integrand $$f_\eps := p(\tvr_\eps) \nabla^2 \Delta^{-1} [\mathbf{1}_D \tvr_\eps], $$
and recall that the pressure $p(\tvr_\eps)$ is bounded uniformly in the space $ L(\log L)^{1+\kappa}(Q_T)\subset  L \log L(Q_T)$, where $Q_T = (0,T) \times D$. Let $A \subset Q_T$ be a measurable set and $A(t) = \{x \in D : (t,x) \in A\}$ its time-slices. By the Fefferman-Stein inequality (see Proposition~\ref{prop:Maxfcts}), we have
\begin{align*}
    \iint_A |f_\eps| \dd x \dd t &\leq \int_0^T \left( \int_{\R^d} |p(\tvr_\eps) \mathbf{1}_{A(t)}| \, |\nabla^2 \Delta^{-1} [\mathbf{1}_D \tvr_\eps]| \dd x \right) \dd t \\
    &\lesssim \int_0^T \| M(p(\tvr_\eps) \mathbf{1}_{A(t)}) \|_{L^1(D)} \| M^\sharp (\nabla^2 \Delta^{-1} [\mathbf{1}_D \tvr_\eps]) \|_{L^\infty(D)} \dd t \\
    &\lesssim \int_0^T \left( \int_D M(p(\tvr_\eps) \mathbf{1}_{A(t)}) \dd x \right) \dd t,
\end{align*}
where we have used the estimate $\sup_{t\in (0,T)} \|\nabla^2 \Delta^{-1} [\mathbf{1}_D \tvr_\eps]\|_{\BMO(D)} \lesssim \|\tvr_\eps\|_{L^\infty(0,T; L^\infty(D))} \lesssim 1$. Here, $M$ and $M^\sharp$ denote the Hardy-Littlewood maximal and sharp maximal operators, respectively, as defined in Appendix~\ref{sec:AppA}.

Applying the Stein-type $L \log L$ estimate from Proposition~\ref{prop:Maxfcts} on each time-slice $A(t)$ and integrating in time, we obtain
{\small
\begin{align*}
\iint_A |f_\eps| \dd x \dd t &\lesssim \int_0^T |A(t)| \left(1 - \log \frac{|A(t)|}{|D|} \right) \dd t \quad \\
&\quad + \int_0^T  \lr{ \int_{A(t)} p(\tvr_\eps) \log(1 + p(\tvr_\eps)) \dd x } \lr{1 + \left| \log \lr{\int_{A(t)} p(\tvr_\eps) \log(1 + p(\tvr_\eps)) \dd x} \right|}\dd t\\
&=: A_1+A_2.
\end{align*}
}
We have to ensure that the right-hand side vanishes uniformly in $\eps$ as $|A| \to 0$. This will be a consequence of the following estimates:
\begin{itemize}[leftmargin=*]
\item \textbf{Term $A_1$:} The function $h(s) = s (1 - \log (s / |D|))$ is concave for $s > 0$. Therefore, by Jensen's inequality and $\int_0^T |A(t)| \dd t = |A|$, we have $$A_1 = \int_0^T h(|A(t)|) \dd t \leq T h(|A|/T),$$
which depends only on the total measure $|A|$ and thus vanishes as $|A| \to 0$. \\

\item \textbf{Term $A_2$:} We define $I_\eps(t) = \int_{A(t)} \Psi(p(\tvr_\eps)) \dd x$, where $\Psi(s) = s \log(1+s)$. The first improved pressure estimate (Step~2 in Lemma~\ref{imprPress}) only places $\Psi(p(\tvr_\eps))$ in $L^1((0,T) \times D)$, which is insufficient to ensure that the term $I_\eps(1 + |\log I_\eps|)$ vanishes uniformly. However, utilizing the established $L (\log L)^{1+\kappa}$ estimate
\begin{align*}
    \int_0^T \int_{D} p(\tvr_\eps) \left[ \log(1 + p(\tvr_\eps)) \right]^{1+\kappa} \dd x \dd t \leq C < \infty
\end{align*}
from Lemma~\ref{imprPress}, the sequence $\{ \Psi(p(\tvr_\eps)) \}_{\eps \in (0,1)}$ is equi-integrable in $Q_T$ by the De la Vall\'{e}e-Poussin theorem, as it is bounded in the Orlicz space $L \lr{\log L}^\kappa(Q_T)$ with $\kappa > 0$. This implies
\begin{align*}
    \Theta_\eps(A) := \int_0^T I_\eps(t) \dd t = \iint_A \Psi(p(\tvr_\eps)) \dd x \dd t \to 0 \quad \text{as } |A| \to 0 \text{ uniformly in } \eps.
\end{align*}
To close the estimate for $A_2$, we apply Jensen's inequality to the function $g(s) = s (1 + |\log s|)$. Since $g$ is concave for $s \in (0, 1)$, we consider $|A|$ sufficiently small such that $\frac{1}{T} \Theta_\eps(A) < 1$ for all $\eps \in (0,1)$, yielding
\begin{align*}
    A_2 = \int_0^T g(I_\eps(t)) \dd t \leq T g\left( \frac{1}{T} \int_0^T I_\eps(t) \dd t \right) = T g\left( \frac{\Theta_\eps(A)}{T} \right) \to 0
\end{align*}
since $\Theta_\eps(A) \to 0$ uniformly in $\eps$ as $|A| \to 0$, and $\lim_{s \to 0} s \log s = 0$.
\end{itemize}

Combining $A_1$ and $A_2$, we conclude that $$ \lim_{|A| \to 0} \int \int_A |f_\eps| \dd x \dd t = 0$$ uniformly in $\eps$, thereby establishing the required equi-integrability of the sequence $\{f_\eps\}_{\eps \in (0,1)}$. Moreover, since $\psi(t)$ and $\zeta(x)$ are bounded, the sequence $\psi \zeta f_\eps$ is equi-integrable as well. Applying Lemma~\ref{lem:equiConv} with $E_\eps = \text{supp}(\Phi_\eps - \mathbb{I})$ -- whose measure vanishes as $\eps \to 0$ -- we conclude
\begin{align*}
    \lim_{\eps \to 0} \| (\Phi_\eps - \mathbb{I}) \psi \zeta f_\eps \|_{L^1((0,T) \times D)} = 0,
\end{align*}
thus proving \eqref{eq1}.


\subsection{Proof of Lemma~\ref{evf-lem}}
To complete the proof, with a slight rearrangement of terms obtained after the limit passage in \eqref{conv1bd} and following the arguments of \cite{FeireislZhang2010}, we obtain
\begin{align*}
	&\int_0^T \psi(t) \lr{ \int_D \zeta(x) \overline{p(\rho)\rho} \dd x } \dd t  
	- \int_0^T \psi(t) \lr{ \int_D \zeta(x) \lr{ \frac{2(d-1)}{d} \mu+ \mu_1 } \overline{\rho \div \vu} \dd x } \dd t\\ 
	&\quad + \int_0^T \partial_t \psi(t) \lr{ \int_D \rho \vu \cdot \zeta(x) \nabla \Delta^{-1}[\mathbf{1}_{D} \rho]  \dd x } \dd t 
	- \int_0^T \psi(t) \lr{ \int_D \rho \vu \cdot  \zeta(x) \nabla \Delta^{-1} \div (\rho \vu) \dd x } \dd t \\ 
    &\quad + \int_0^T \psi(t) \lr{ \int_D \rho \vu \otimes \vu : \nabla \zeta(x) \otimes \nabla \Delta^{-1}[\mathbf{1}_{D} \rho] \dd x } \dd t\\
	&\quad + \int_0^T \psi(t) \lr{ \int_D \rho \vu \otimes \vu :  \zeta(x) \nabla^2 \Delta^{-1}[\mathbf{1}_{D} \rho] \dd x } \dd t\\ 
	&\quad + \int_0^T \psi(t) \lr{ \int_D \overline{p(\rho)} \mathbb{I} : \nabla \zeta(x) \otimes \nabla \Delta^{-1}[\mathbf{1}_{D} \rho] \dd x } \dd t \\
	&\quad - \int_0^T \psi(t) \lr{ \int_D \mathbb{S}(\nabla \vu) : \nabla \zeta(x) \otimes \nabla \Delta^{-1}[\mathbf{1}_{D} \rho] \dd x } \dd t\\
	&\quad - \int_0^T \psi(t) \lr{ \int_D \rho \vc f \cdot  \zeta(x) \nabla \Delta^{-1}[\mathbf{1}_{D} \rho] \dd x } \dd t =0.
\end{align*}

On the other hand, considering 
\[
\phi(t,x) = \psi(t) \zeta(x) \nabla \Delta^{-1}[\mathbf{1}_{D} \varrho]
\] 
as test function in the limiting momentum equation \eqref{wkMom-lim}, it holds that
\begin{align*}
	&\int_0^T \psi(t) \lr{ \int_D \zeta(x) \overline{p(\rho)} \, \rho \dd x } \dd t  
	- \int_0^T \psi(t) \lr{ \int_D \zeta(x) \lr{\frac{2(d-1)}{d} \mu + \mu_1 } \rho \div \vu  \dd x } \dd t\\ 
	&\quad + \int_0^T \partial_t \psi(t) \lr{ \int_D \rho \vu \cdot \zeta(x) \nabla \Delta^{-1}[\mathbf{1}_{D} \rho ] \dd x } \dd t 
	- \int_0^T \psi(t) \lr{ \int_D \rho \vu \cdot  \zeta(x) \nabla \Delta^{-1} \div (\rho \vu) \dd x } \dd t \\ 
	&\quad + \int_0^T \psi(t) \lr{ \int_D \rho \vu \otimes \vu : \nabla \zeta(x) \otimes \nabla \Delta^{-1}[\mathbf{1}_{D} \rho] \dd x } \dd t\\
    &\quad + \int_0^T \psi(t) \lr{ \int_D \rho \vu \otimes \vu :  \zeta(x) \nabla^2 \Delta^{-1}[\mathbf{1}_{D} \rho] \dd x } \dd t\\ 
	&\quad + \int_0^T \psi(t) \lr{ \int_D \overline{p(\rho)} \mathbb{I} : \nabla \zeta(x) \otimes \nabla \Delta^{-1}[\mathbf{1}_{D} \rho] \dd x } \dd t \\
	&\quad - \int_0^T \psi(t) \lr{ \int_D \mathbb{S}(\nabla \vu) : \nabla \zeta(x) \otimes \nabla \Delta^{-1}[\mathbf{1}_{D} \rho] \dd x } \dd t\\
	&\quad - \int_0^T \psi(t) \lr{ \int_D \rho \vc f \cdot  \zeta(x) \nabla \Delta^{-1}[\mathbf{1}_{D} \rho] \dd x } \dd t = 0.
\end{align*}

Finally, by comparing the last two identities, we deduce the weak form of the effective viscous flux identity \eqref{evf-id-weak}, thus completing the proof of Theorem~\ref{thm1}.


\appendix
\section{Important function spaces and their properties}\label{sec:AppB}
In this section, we collect essential definitions and properties of the functional spaces required for our analysis. In particular, we focus on the hierarchy of Orlicz spaces and the BMO-Hardy framework.

\subsection{Logarithmic spaces, BMO spaces, and duality}
\begin{itemize}[leftmargin=*]
    \item \textbf{The $L \log L$ and $L(\log L)^{1+\kappa}$ spaces:} 
    For a bounded domain $D \subset \mathbb{R}^d$, we consider the Orlicz space $L \log L(D)$ associated with the Young function $\Psi(s) = s \log(1+s)$. This space is separable because $\Psi$ satisfies the $\Delta_2$-condition\footnote{For the definition of the $\Delta_2$-condition, see \cite[Chapter~8]{AF2003}; in fact $\Psi(2s) \leq 4 \Psi(s)$ for all $s>0$.} for all $s > 0$. For $\kappa > 0$, we further define the refined space $L(\log L)^{1+\kappa}(D)$ via the growth function $\Psi_{1+\kappa}(s) = s [\log(1+s)]^{1+\kappa}$. These spaces provide a hierarchy of integrability strictly stronger than $L^1(D)$, satisfying the continuous embeddings
    $$ L^p(D) \hookrightarrow L(\log L)^{1+\kappa}(D) \hookrightarrow L \log L(D) \hookrightarrow L^1(D), \quad \text{for any } p > 1. $$

    \item \textbf{The space $\BMO$:}
The space of functions of Bounded Mean Oscillation, $\BMO(\mathbb{R}^n)$, consists of locally integrable functions $f \in L_{loc}^1(\mathbb{R}^d)$ such that the mean oscillation over all cubes $Q \subset \mathbb{R}^n$ is uniformly bounded:
\begin{equation*}
    \|f\|_{\BMO(\mathbb{R}^n)} := \sup_{Q \subset \mathbb{R}^n} \frac{1}{|Q|} \int_Q |f(x) - f_Q| \dd x < \infty,
\end{equation*}
where $f_Q = \frac{1}{|Q|} \int_Q f(y) \dd y$ denotes the average of $f$ over the cube $Q$. For a bounded domain $D$, we define $\BMO(D)$ as the space of functions that are restrictions to $D$ of functions in $\BMO(\mathbb{R}^d)$, equipped with the semi-norm
\begin{equation*}
    \|f\|_{\BMO(D)} = \inf \{ \|g\|_{\BMO(\mathbb{R}^d)} : g|_D = f \}.
\end{equation*}
The space $(\BMO(D), \| \cdot \|_{\BMO(D)})$ becomes a Banach space when identifying constant functions.
    \item \textbf{Exponential spaces, duality with Logarithmic spaces, and embeddings:}
    The dual of $L \log L(D)$ is identified as \textit{the exponential Orlicz space} $\exp L(D)$, while the dual of $L(\log L)^{1+\kappa}(D)$ is the space $\exp L^{1/(1+\kappa)}(D)$, defined by the Young function $\Psi_{1+\kappa}(s) = \exp(s^{1/(1+\kappa)}) - 1$. For further properties of these spaces, we refer to \cite[Section~2]{PokornySkrisovsky2021b}. By the John-Nirenberg inequality \cite{JohnNirenberg1961} (see also \cite[Chapter~6]{Duoandikoetxea2001}), any function $f \in \BMO(D)$ satisfies an exponential decay of its distribution function, leading to the continuous embeddings
    $$ \BMO(D) \hookrightarrow \exp L(D) \hookrightarrow \exp L^{1/(1+\kappa)}(D). $$

    \item \textbf{Hardy Space and Duality:} Finally, we recall the duality relationship involving the space of bounded mean oscillation and the Hardy space $\mathcal{H}^1(\mathbb{R}^d)$ (see \cite[Chapter~5.3]{Duoandikoetxea2001}). A crucial property in this context is that $\mathcal{H}^1(\mathbb{R}^d)$ is a \textit{separable} Banach space, which allows for the identification of its dual as
\begin{equation*}
    (\mathcal{H}^1(\mathbb{R}^d))^* = \BMO(\mathbb{R}^d).
\end{equation*}
In the case of a bounded Lipschitz domain $D \subset \mathbb{R}^d$, this duality remains valid for the corresponding spaces $\mathcal{H}^1(D)$ and $\BMO(D)$ defined via atoms and restrictions, respectively. The separability of $\mathcal{H}^1(D)$ is particularly significant for our analysis, as it ensures the well-definedness of the time-dependent dual space $L^\infty_{w-(\ast)}(0,T; \BMO(D))$.
\end{itemize}


\subsection{Time dependent spaces}
We consider $D$ to be $\mathbb{R}^d$ or a sufficiently smooth bounded domain in $\mathbb{R}^d$. Then we have:
\begin{itemize}[leftmargin=*]
    \item \textbf{Embeddings 1:} 
As a consequence of the definitions, Fubini's theorem, and Jensen's inequality, we have the following continuous embeddings for the logarithmic scale:
\begin{equation}\label{spatio_temp_LlogL}
    L\log L( (0,T) \times D) \hookrightarrow L^1(0,T; L\log L(D)) \hookrightarrow L^1((0,T) \times D).
\end{equation}
Similarly, by analogous arguments to \eqref{spatio_temp_LlogL}, the spatio-temporal embedding also holds for the refined scale:
\begin{equation*}
    L(\log L)^{1+\kappa}((0,T) \times D) \hookrightarrow L^1(0,T; L(\log L)^{1+\kappa}(D)).
\end{equation*}

    \item \textbf{Time-Dependent Dualities 1:} Following \cite[Chapter 2]{Chaudhuri2021}, we identify the time-dependent dual space as
\begin{equation}\label{loglexpl}
     L^\infty_{w-(\ast)}(0,T; \exp L^{1/(1+\kappa)}(D)) = [L^1(0,T; L (\log L)^{1+\kappa}(D))]^*. 
\end{equation}
Since the spatial Orlicz space $L (\log L)^{1+\kappa}(D)$ is separable, the corresponding Bochner space is separable, too. This duality \eqref{loglexpl} is essential for ensuring that the terms involving pressure are well-defined within our framework.
  \item \textbf{Time dependent Dualities 2:} A similar relation holds for the Hardy and $\BMO$ spaces. Since the Hardy space $\mathcal{H}^1(D)$ is separable, we have the duality
\begin{equation*}
     L^\infty_{w-(\ast)}(0,T; \BMO(D)) = [L^1(0,T; \mathcal{H}^1(D))]^*.
\end{equation*}
\item \textbf{Embeddings 2:}  Corresponding to the spatial embedding $\BMO(D) \hookrightarrow \exp L(D)$, it naturally extends to the following chain of spatio-temporal embeddings for any $\kappa\geq 0$:
\begin{equation}\label{time_BMO_exp}
    L^\infty_{w-(\ast)}(0,T; \BMO(D)) \hookrightarrow L^\infty_{w-(\ast)}(0,T; \exp L(D)) \hookrightarrow L^\infty_{w-(\ast)}(0,T; \exp L^{\frac{1}{1+\kappa}}(D)) .
\end{equation}

\end{itemize}

\subsection{Important results in $\BMO$, $L \log L$, and $L (\log L)^{1+\kappa}$}\label{sec:AppA}

We begin with the classical definition of the maximal function(s) from \cite[Chapters~2--6]{Duoandikoetxea2001}:
\begin{defin}
    Let $f: \R^d \to \R$ be a locally integrable function.
    \begin{itemize}[leftmargin=*]
        \item The \emph{Hardy-Littlewood maximal function} is defined as
    \begin{align*}
        Mf(x) := \sup_{r>0} \frac{1}{|B_r(x)|} \int_{B_r(x)} |f(y)| \dd y,
    \end{align*}
    where $B_r(x)$ is the ball with radius $r>0$ and midpoint $x$.

    \item The \emph{Fefferman-Stein sharp maximal function} is defined as
    \begin{align*}
        M^\sharp f(x) := \sup_{B \ni x} \frac{1}{|B|} \int_B |f(y) - f_B| \dd y,
    \end{align*}
    where the supremum is taken over all balls $B$ containing $x$, and $f_B := |B|^{-1} \int_B f(z) \dd z$ is the mean value of $f$ over $B$.

    \item The space of \emph{bounded mean oscillation} $\BMO(\R^d)$ is defined as all functions $f \in L_{loc}^1(\R^d)$ with $M^\sharp f \in L^\infty(\R^d)$. It becomes a Banach space with the norm
    \begin{align*}
        \|f\|_{\BMO} := \|M^\sharp f\|_{L^\infty}
    \end{align*}
    when identifying constant functions. For a bounded domain $D \subset \R^d$, we define $\BMO(D)$ as all functions $f \in L^1(D)$ such that there is a function $g \in \BMO(\R^d)$ with $g |_D = f$.
    \item The $L \log L$ space (sometimes called \emph{Zygmund space}) is the space of all measurable functions $f$ such that
    \begin{align*}
        f \log(1+|f|) \in L^1.
    \end{align*}
    It becomes a Banach space with the norm
    \begin{align*}
        \|f\|_{L \log L} := \|f \log(1+|f|)\|_{L^1}.
    \end{align*}
    \end{itemize}
\end{defin}

The fundamental advantage for us is that $\nabla^2 \Delta^{-1}$ maps $L^\infty$ to $\BMO \supset L^\infty$ such that we can work with this larger space instead. We report some fundamental connections between the above definitions:
\begin{prop}\label{prop:Maxfcts}
Let $f \in L_{loc}^1(\R^d)$. Then there holds:
    \begin{enumerate}[leftmargin=*]
    
        \item Let $f \in L \log L$ and $g \in \BMO$. Then $fg \in L^1$.

        \item (weak-(1,1) estimate) Let $f \in L^1$. Then, for any $\lambda>0$, we have
        \begin{align*}
            |\{Mf > \lambda \}| \lesssim \frac{1}{\lambda} \int_{\R^d} |f| \dd x.
        \end{align*}

        \item (Fefferman-Stein estimate) If $f \in L \log L$, then $Mf \in L^1$. Moreover, if $f \in L \log L$ and $g \in \BMO$, then
        \begin{align*}
            \int_{\R^d} |fg| \dd x \lesssim \int_{\R^d} Mf \cdot M^\sharp g \dd x, && \text{in particular,} && \|fg\|_{L^1} \lesssim \| M f\|_{L^1} \| M^\sharp g\|_{L^\infty}.
        \end{align*}

        \item Let $D \subset \R^d$ be a bounded domain. For any $1<p<\infty$, we have $L^p \subset L \log L \subset L^1$, and $L^\infty \subset \BMO \subset L^p$, where all inclusions are strict.
        
        \item (Stein-type $L \log L$ estimate) Let $D \subset \R^d$ be a bounded domain, $E \subset D$ be measurable, $f \in L \log L$ and ${\rm supp}  f \subset E$. Then
        \begin{align*}
            \int_D Mf \dd x \lesssim |E|\lr{ 1-\log \frac{|E|}{|D|} } + \lr{ \int_E |f| \log(1+|f|) \dd x } \lr{ 1 + \left| \log \int_E |f| \log(1+|f|) \dd x \right| }.
        \end{align*}
    \end{enumerate}
\end{prop}
\begin{proof}
    The first four statements are classical in the literature, see, e.g., \cite{BennettSharpley1988, Duoandikoetxea2001}. Let us therefore just prove the last one, which is crucial for us. We first use the layer-cake formula to express
    \begin{align*}
        \int_D Mf \dd x = \int_0^\infty |\{Mf > \lambda\}| \dd \lambda = \int_0^1 |\{Mf > \lambda\}| \dd \lambda + \int_1^\infty |\{Mf > \lambda\}| \dd \lambda =: I_1 + I_2.
    \end{align*}
    To estimate $I_1$, we split
    \begin{align*}
        f = f_1 + f_2 := f \mathbf{1}_{|f| \leq 1} + f \mathbf{1}_{|f| > 1}
    \end{align*}
    and see that by definition, we have $\{ Mf > \lambda\} \subset \{Mf_1 > \lambda/2 \} \cup \{Mf_2 > \lambda/2\}$. Furthermore, obviously, $|\{Mf_1 > \lambda/2\}| \leq |D|$, and the weak-(1,1) estimate together with $|f_1| \leq 1$ and ${\rm supp} f_1 \subset E$ yields
    \begin{align*}
        |\{Mf_1 > \lambda/2\}| \lesssim \min\{|D|, \lambda^{-1} |E|\}.
    \end{align*}
    In turn, by seeing that $|Mf_1| \leq 1$ and splitting the integral at the balance point $\lambda_0 = |E| / |D|$, we estimate
    \begin{align*}
        \int_0^1 |\{Mf_1 > \lambda/2\}| \dd \lambda \lesssim \int_0^1 \min \{|D|, \lambda^{-1} |E| \} \dd \lambda \lesssim |E| - |E| \log \frac{|E|}{|D|}.
    \end{align*}
    
    For $f_2$, we have similarly $$|\{Mf_2 > \lambda/2\}| \lesssim \min\{|D|, \lambda^{-1} \int_D |f_2| \dd x \} = \min\{|D|, \lambda^{-1} \int_{E \cap \{|f|>1\}} |f| \dd x\} =: \min\{|D|, \lambda^{-1} Z\} .$$ Then, since $s \leq 2 s \log(1+s)$ if $s \geq 1$, we have
    \begin{align*}
        \int_0^1 |\{Mf_2 > \lambda/2\}| \dd \lambda &\lesssim \int_0^1 \min\{|D|, \lambda^{-1} Z\} \dd \lambda = Z - Z \log \frac{Z}{|D|} \\
        &\lesssim \lr{ \int_E |f| \log(1+|f|) \dd x } \lr{ 1 + \log |D| + \left| \log \int_{E \cap \{|f|>1\}} |f| \dd x \right| } \\
        &\lesssim \lr{ \int_E |f| \log(1+|f|) \dd x } \lr{ 1 + \left| \log \int_E |f| \log(1+|f|) \dd x \right| },
    \end{align*}
    such that in total
    \begin{align*}
        I_1 \lesssim |E| - |E| \log \frac{|E|}{|D|} + \lr{ \int_E |f| \log(1+|f|) \dd x } \lr{ 1 + \left| \log \int_E |f| \log(1+|f|) \dd x \right| } .
    \end{align*}

    For $\lambda > 1$, we split again
    \begin{align*}
        f = f_1 + f_2 := f \mathbf{1}_{|f| \leq \lambda/2} + f \mathbf{1}_{|f| > \lambda/2}.
    \end{align*}
    Then $Mf \leq Mf_1 + Mf_2$ and $\{Mf > \lambda\} \subset \{Mf_1 > \lambda/2\} \cup \{Mf_2 > \lambda/2\} = \{Mf_2 > \lambda/2\}$ and by the weak-(1,1) estimate, we have
    \begin{align*}
        |\{Mf > \lambda\}| \lesssim \frac{1}{\lambda} \int_{|f|>\lambda/2} |f| \dd x.
    \end{align*}
    Thus, using Fubini's theorem and the fact that ${\rm supp}  f \subset E$,
    \begin{align*}
        I_2 = \int_1^\infty |\{Mf > \lambda\}| &\lesssim \int_1^\infty \frac{1}{\lambda} \int_{|f| >\lambda/2} |f| \dd x \dd \lambda = \int_E |f| \int_1^{2|f|} \frac{1}{\lambda} \dd \lambda \dd x \\
        &\lesssim \int_E |f| \log(1+|f|) \dd x,
    \end{align*}
    where in the last inequality we used that for any $s > 0$, we have $\log(2 s) \leq 2 \log(1+s)$. Combining the above estimates yields the desired.
\end{proof}


\subsection{The $L\log L \times \BMO$ product and the failure of equi-integrability} \label{lologl-bmo-ce}

Two explain why the $L (\log L)^{1+\kappa}$-bound is strictly needed, we consider two sequences, $\{f_n\}_{n\in\mathbb{N}}$ and $\{g_n\}_{n\in\mathbb{N}}$, which are uniformly bounded in $L \log L(D)$ and $\BMO(D)$, respectively. The duality between these spaces ensures that the product sequence $\{f_n g_n\}_{n\in\mathbb{N}}$ is uniformly bounded in $L^1(D)$. However, since $L^1$ is not a reflexive space, this boundedness is insufficient to rule out the formation of concentration measures (Dirac-like singularities) in the limit. In what follows, we construct an explicit example to demonstrate this phenomenon.\\

Let us consider the domain $\Omega=(0,1)$ and let $I_n = (0, e^{-n})$. We define two sequences of function $\{f_n\}_{n\in \mathbb{N}}$ and $\{g_n\}_{n\in \mathbb{N}}$ by
    \begin{align*}
    f_n(x) = \frac{e^n}{n} \mathbf{1}_{I_n}(x) , && g_n(x) \equiv g(x) = \log(1/x) .
    \end{align*}
    From \cite[Chapter~6]{Duoandikoetxea2001}, we know that $\displaystyle g \in \BMO(0,1)$. Furthermore, we calculate 
$$\int_0^1 f_n \log(1+f_n) \dd x = |I_n|  \frac{e^n}{n} \log\left(1 + \frac{e^n}{n}\right)=  \frac{1}{n} \log\left(1 + \frac{e^n}{n}\right) .$$
Using $2e^n > e^n + n > e^n$ for each $n \geq 1$, it holds
\begin{align*}
   \left(1 - \frac{\log n}{n}\right) \leq  \int_0^1 f_n \log(1+f_n) \dd x \leq \left(\frac{\log 2}{n} + 1 - \frac{\log n}{n}\right).
\end{align*}
Thus, $f_n$ is uniformly bounded in $L \log L$. Computing the integral of the product $f_n g$ gives 
\begin{align*}
   \int_0^1 |f_n g| \dd x = \int_0^1 f_n g \dd x = \int_0^{e^{-n}} \frac{e^n}{n} \cdot \log(1/x) \dd x .
\end{align*}
We use the change of variables $u = \log(1/x)$ and partial integration to obtain 
\begin{align*}
    \int_0^1 |f_n g| \dd x = \frac{e^n}{n} \int_n^\infty u e^{-u} \dd u = \frac{e^n}{n} \left[ -e^{-u}(u+1) \right]_n^\infty = \frac{e^n}{n} \left( e^{-n}(n+1) \right) = \frac{n+1}{n} >1 .
\end{align*}

Since we can find a set $E=I_n$ with arbitrarily small measure where the integral of $f_n g$ stays above 1, the sequence $\{ f_n g \}_{n \in \N}$ cannot be equi-integrable.

\begin{rem}
    It is worth to note that $\{f_n\}_{n\in \mathbb{N}} $ is not uniformly bounded in $L(\log L)^{1+\kappa}$ for any $\kappa >0$, since for any $n \geq 1$, it holds
 \[
\log \left(1+\frac{e^n}{n}\right) \geq n - \log n \geq \frac{n}{2}
 \Rightarrow 
\left[\log\lr{1+\frac{e^n}{n}}\right]^{1+\kappa} \geq \lr{\frac{n}{2}}^{1+\kappa}.
\]
Thus, we obtain 
$$\int_0^1 f_n |\log(1+f_n) |^{1+\kappa}\dd x = |I_n|  \frac{e^n}{n} \left[\log\left(1 + \frac{e^n}{n}\right)\right]^{1+\kappa}=  \frac{1}{n} \left[\log\left(1 + \frac{e^n}{n}\right)\right]^{1+\kappa} \geq \frac{(n/2)^{\kappa}}{2}.$$
We also note that our pressure sequence $\{p(\tvr_\eps)\}_{\eps \in (0,1)}$ \emph{does} have a uniform bound in $L(\log L)^{1+\kappa}$ for any $\kappa \in (0, \infty)$, hence, it does indeed not develop such pathological Dirac-like concentrations.
\end{rem}

\section*{Acknowledgement}
	\textit{The authors thank Lubo\v{s} Pick and Piotr Mucha for helpful discussions about $L \log L$ and $\BMO$. The work of N.C. is supported by the NAWA ULAM grant BPN/SEL/2025/1/00005/U /00001. F. O. has been supported by the Primus grant PRIMUS 26/SCI/026.}

\bibliographystyle{amsalpha}
\bibliography{Lit}

\providecommand{\bysame}{\leavevmode\hbox to3em{\hrulefill}\thinspace}
\providecommand{\MR}{\relax\ifhmode\unskip\space\fi MR }
\providecommand{\MRhref}[2]{%
  \href{http://www.ams.org/mathscinet-getitem?mr=#1}{#2}
}
\providecommand{\href}[2]{#2}
\begin{thebibliography}{HNO26}

\bibitem[AF03]{AF2003}
Robert~A. Adams and John J.~F. Fournier, \emph{Sobolev spaces}, second ed., Pure and Applied Mathematics (Amsterdam), vol. 140, Elsevier/Academic Press, Amsterdam, 2003. \MR{2424078}

\bibitem[All90a]{Allaire1990a}
Gr\'{e}goire Allaire, \emph{Homogenization of the {N}avier--{S}tokes equations in open sets perforated with tiny holes. {I}. {A}bstract framework, a volume distribution of holes}, Arch. Rational Mech. Anal. \textbf{113} (1990), no.~3, 209--259. \MR{1079189}

\bibitem[All90b]{Allaire1990b}
\bysame, \emph{Homogenization of the {N}avier--{S}tokes equations in open sets perforated with tiny holes. {II}. {N}oncritical sizes of the holes for a volume distribution and a surface distribution of holes}, Arch. Rational Mech. Anal. \textbf{113} (1990), no.~3, 261--298. \MR{1079190}

\bibitem[BC24]{BasaricChaudhuri2024}
Danica Basari{\'c} and Nilasis Chaudhuri, \emph{Low mach number limit on perforated domains for the evolutionary navier--stokes--fourier system}, Nonlinearity \textbf{37} (2024), no.~6, 065008.

\bibitem[BO22]{BellaOschmann2022}
Peter Bella and Florian Oschmann, \emph{Homogenization and low {M}ach number limit of compressible {N}avier-{S}tokes equations in critically perforated domains}, Journal of Mathematical Fluid Mechanics \textbf{24} (2022), no.~3, 1--11.

\bibitem[BO23]{BellaOschmann2023}
\bysame, \emph{Inverse of divergence and homogenization of compressible {N}avier--{S}tokes equations in randomly perforated domains}, Arch. Ration. Mech. Anal. \textbf{247} (2023), no.~2, 14.

\bibitem[BOP26]{BasaricOschmannPan2026}
Danica Basari{\'c}, Florian Oschmann, and Jiaojiao Pan, \emph{Qualitative derivation of a density-dependent incompressible {D}arcy law}, Science China Mathematics (2026), 1--24.

\bibitem[BPZ14]{BreschPerrinZatorska2014}
Didier Bresch, Charlotte Perrin, and Ewelina Zatorska, \emph{Singular limit of a {N}avier--{S}tokes system leading to a free/congested zones two-phase model}, Comptes Rendus Mathematique \textbf{352} (2014), no.~9, 685--690.

\bibitem[Bra24]{Bravin2024}
Marco Bravin, \emph{Ad hoc test functions for homogenization of compressible viscous fluid with application to the obstacle problem in dimension two}, Journal of Evolution Equations \textbf{24} (2024), no.~4, 84.

\bibitem[BS88]{BennettSharpley1988}
Colin Bennett and Robert~C Sharpley, \emph{Interpolation of operators}, vol. 129, Academic press, 1988.

\bibitem[Cha21]{Chaudhuri2021}
Nilasis Chaudhuri, \emph{Qualitative properties of solutions to partial differential equations arising in fluid dynamics}, ProQuest LLC, Ann Arbor, MI, 2021, Thesis (Ph.D.)--Technische Universit\"at Berlin.

\bibitem[DFL17]{DieningFeireislLu2017}
Lars Diening, Eduard Feireisl, and Yong Lu, \emph{The inverse of the divergence operator on perforated domains with applications to homogenization problems for the compressible {N}avier--{S}tokes system}, ESAIM: Control, Optimisation and Calculus of Variations \textbf{23} (2017), no.~3, 851--868.

\bibitem[DH13]{DegondHua2013}
Pierre Degond and Jiale Hua, \emph{Self-organized hydrodynamics with congestion and path formation in crowds}, Journal of Computational Physics \textbf{237} (2013), 299--319.

\bibitem[DHN11]{DegondHuaNavoret2011}
Pierre Degond, Jiale Hua, and Laurent Navoret, \emph{Numerical simulations of the {E}uler system with congestion constraint}, Journal of Computational Physics \textbf{230} (2011), no.~22, 8057--8088.

\bibitem[Duo01]{Duoandikoetxea2001}
Javier Duoandikoetxea, \emph{Fourier analysis}, vol.~29, American Mathematical Soc., 2001.

\bibitem[FL15]{FeireislLu2015}
Eduard Feireisl and Yong Lu, \emph{Homogenization of stationary {N}avier--{S}tokes equations in domains with tiny holes}, Journal of Mathematical Fluid Mechanics \textbf{17} (2015), no.~2, 381--392.

\bibitem[FLM16]{FeireislLuMalek2016}
Eduard Feireisl, Yong Lu, and Josef M{\'a}lek, \emph{On {PDE} analysis of flows of quasi-incompressible fluids}, ZAMM Z. Angew. Math. Mech. \textbf{96} (2016), no.~4, 491--508.

\bibitem[FN09]{FeireislNovotny2009singlim}
Eduard Feireisl and Anton{\'\i}n Novotn{\'y}, \emph{Singular limits in thermodynamics of viscous fluids}, vol.~2, Springer, 2009.

\bibitem[FZ10]{FeireislZhang2010}
Eduard Feireisl and Ping Zhang, \emph{Quasi-neutral limit for a model of viscous plasma}, Archive for rational mechanics and analysis \textbf{197} (2010), no.~1, 271--295.

\bibitem[GH19]{GiuntiHoefer2019}
Arianna Giunti and Richard~Matthias H{\"o}fer, \emph{Homogenisation for the {S}tokes equations in randomly perforated domains under almost minimal assumptions on the size of the holes}, Ann. Inst. H. Poincar\'{e} Anal. Non Lin\'{e}aire \textbf{36} (2019), no.~7, 1829--1868. \MR{4020526}

\bibitem[Giu21]{Giunti2021}
Arianna Giunti, \emph{Derivation of {D}arcy’s law in randomly perforated domains}, Calculus of Variations and Partial Differential Equations \textbf{60} (2021), no.~5, 1--30.

\bibitem[HHR26]{HoeferHuebner2026}
Richard~Matthias H{\"o}fer and Eleni H{\"u}bner-Rosenau, \emph{Homogenization of the {N}avier-{S}tokes equations in a randomly perforated domain in the inviscid limit}, arXiv preprint arXiv:2604.14792 (2026).

\bibitem[HKS21]{HoeferKowalczykSchwarzacher2021}
Richard~Matthias H{\"o}fer, Karina Kowalczyk, and Sebastian Schwarzacher, \emph{Darcy’s law as low {M}ach and homogenization limit of a compressible fluid in perforated domains}, Mathematical Models and Methods in Applied Sciences \textbf{31} (2021), no.~09, 1787--1819.

\bibitem[HLO26]{HoeferLuOschmann2026}
Richard~Matthias H{\"o}fer, Yong Lu, and Florian Oschmann, \emph{Qualitative/quantitative homogenization of some non-{N}ewtonian flows in perforated domains}, Mathematische Annalen \textbf{395} (2026), no.~2, 39.

\bibitem[HNO26]{HoeferNecasovaOschmann2026}
Richard~Matthias H{\"o}fer, {\v{S}}{\'a}rka Ne{\v{c}}asov{\'a}, and Florian Oschmann, \emph{Quantitative homogenization of the compressible {N}avier--{S}tokes equations towards {D}arcy's law}, Annales de l'Institut Henri Poincaré C, Analyse Non Linéaire \textbf{43} (2026), no.~3, 669--709.

\bibitem[JN61]{JohnNirenberg1961}
Fritz John and Louis Nirenberg, \emph{On functions of bounded mean oscillation}, Communications on Pure and Applied Mathematics \textbf{14} (1961), no.~3, 415--426.

\bibitem[Lio98]{Lions1998}
Pierre-Louis Lions, \emph{Mathematical topics in fluid mechanics. {V}ol. 2}, Oxford Lecture Series in Mathematics and its Applications, vol.~10, The Clarendon Press, Oxford University Press, New York, 1998, Compressible models, Oxford Science Publications. \MR{1637634}

\bibitem[LPY25]{LuPanYang2025}
Yong Lu, Jiaojiao Pan, and Peikang Yang, \emph{Homogenization of inhomogeneous incompressible navier-stokes equations in domains with very tiny holes}, arXiv preprint arXiv:2501.05734 (2025).

\bibitem[LQ24]{LuQian2024}
Yong Lu and Zhengmao Qian, \emph{Homogenization of some evolutionary non-{N}ewtonian flows in porous media}, Journal of Differential Equations \textbf{411} (2024), 619--639.

\bibitem[LS18]{LuSchwarzacher2018}
Yong Lu and Sebastian Schwarzacher, \emph{Homogenization of the compressible {N}avier–-{S}tokes equations in domains with very tiny holes}, Journal of Differential Equations \textbf{265} (2018), no.~4, 1371 -- 1406.

\bibitem[Lu20]{Lu2020}
Yong Lu, \emph{Homogenization of {S}tokes equations in perforated domains: a unified approach}, J. Math. Fluid Mech. \textbf{22} (2020), no.~3, Paper No. 44, 13. \MR{4145838}

\bibitem[LY23]{LuYang2023}
Yong Lu and Peikang Yang, \emph{Homogenization of {E}volutionary {I}ncompressible {N}avier--{S}tokes {S}ystem in {P}erforated {D}omains}, Journal of Mathematical Fluid Mechanics \textbf{25} (2023), no.~1, 4.

\bibitem[NO23]{NecasovaOschmann2023}
{\v{S}}{\'a}rka Ne{\v{c}}asov{\'a} and Florian Oschmann, \emph{Homogenization of the two-dimensional evolutionary compressible {N}avier--{S}tokes equations}, Calculus of Variations and Partial Differential Equations \textbf{62} (2023), no.~6, 184.

\bibitem[NP22]{NecasovaPan2022}
{\v{S}}{\'a}rka Ne{\v{c}}asov{\'a} and Jiaojiao Pan, \emph{Homogenization problems for the compressible {N}avier--{S}tokes system in 2{D} perforated domains}, Mathematical Methods in the Applied Sciences \textbf{45} (2022), no.~12, 7859--7873.

\bibitem[NS04]{Novotny2004}
Antonín Novotn\'y and Ivan Stra\v{s}kraba, \emph{Introduction to the {M}athematical {T}heory of {C}ompressible {F}low}, OUP Oxford, New York, London, 2004.

\bibitem[OP23]{OschmannPokorny2023}
Florian Oschmann and Milan Pokorn{\'y}, \emph{Homogenization of the unsteady compressible {N}avier-{S}tokes equations for adiabatic exponent $\gamma> 3$}, Journal of Differential Equations \textbf{377} (2023), 271--296.

\bibitem[Pan25]{Pan2025}
Jiaojiao Pan, \emph{Homogenization of {N}on-{H}omogeneous {I}ncompressible {N}avier--{S}tokes {S}ystem in {C}ritically {P}erforated {D}omains}, Journal of Mathematical Fluid Mechanics \textbf{27} (2025), no.~2, 1--16.

\bibitem[PS21a]{PokornySkrisovsky2021}
Milan Pokorn{\'y} and Emil Sk{\v{r}}{\'\i}{\v{s}}ovsk{\'y}, \emph{Homogenization of the evolutionary compressible {N}avier--{S}tokes--{F}ourier system in domains with tiny holes}, Journal of Elliptic and Parabolic Equations (2021), 1--31.

\bibitem[PS21b]{PokornySkrisovsky2021b}
\bysame, \emph{Weak solutions for compressible {N}avier--{S}tokes--{F}ourier system in two space dimensions with adiabatic exponent almost one}, Acta Applicandae Mathematicae \textbf{172} (2021), no.~1, 1.

\bibitem[Soa72]{Soave1972}
Giorgio Soave, \emph{Equilibrium constants from a modified {R}edlich-{K}wong equation of state}, Chemical engineering science \textbf{27} (1972), no.~6, 1197--1203.

\bibitem[vdW73]{vanderWaals1873}
Johannes~Diderik van~der Waals, \emph{Over de {C}ontinuiteit van den {G}as - en {V}loeistoftoestand}, vol.~1, Sijthoff, 1873.

\bibitem[WP24]{WiedemannPeter2024}
David Wiedemann and Malte~A Peter, \emph{Homogenisation of the {S}tokes equations for evolving microstructure}, Journal of Differential equations \textbf{396} (2024), 172--209.

\bibitem[WP25]{WiedemannPeter2025}
\bysame, \emph{A {D}arcy law with memory by homogenisation for evolving microstructure}, Journal of Mathematical Analysis and Applications \textbf{546} (2025), no.~2, 129222.

\end{thebibliography}

\end{document}